\documentclass[11pt, a4paper, oneside]{amsart}
\usepackage{graphicx}
\usepackage{amssymb,amsfonts,amsmath,amsthm}
\usepackage{stmaryrd,hyperref,yhmath}
\title{Fronts in dissipative Fermi-Pasta-Ulam-Tsingou chains}

\author{Michael Herrmann} 
\address{Mathematik, Technische Universit\"at Braunschweig,  Germany}
\author{Guillaume James}
\address{Univ. Grenoble Alpes, CNRS, Inria, Grenoble INP, LJK, 38000 Grenoble, France}
\email{guillaume.james@univ-grenoble-alpes.fr}
\author{Karsten Matthies}
\address{Department of Mathematical Sciences, University of Bath, United Kingdom}
\email{k.matthies@bath.ac.uk}
\date{\today}
\DeclareMathOperator{\sinc}{\mathrm{sinc}}

\DeclareMathOperator{\mhIm}{Im}

\newcommand{\iu}{\mathtt{i}}
\newcommand{\mhexp}[1]{{{\mathtt{e}}^{#1}}}

\newcommand{\fspace}[1]{{\mathsf{#1}}}
\newcommand{\fspaceL}{\fspace{L}}
\newcommand{\fspaceH}{\fspace{H}}
\newcommand{\fspaceC}{\fspace{C}}

\newcommand{\ol}[1]{{\overline{#1}}}

\newcommand{\Rset}{{\mathbb{R}}}

\newcommand{\Cset}{{\mathbb{C}}}
\newcommand{\Nset}{{\mathbb{N}}}

\newcommand{\oointerval}[2]{(#1,\,#2)}%

\newcommand{\Do}[1]{{o\at{#1}}}

\newlength{\mhpicDwidth}
\newlength{\mhpicDvsep}
\newlength{\mhpicDhsep}

\newlength{\mhpicPwidth}
\newlength{\mhpicPvsep}
\newlength{\mhpicPhsep}

\newlength{\mhpicWhsep}

\newcommand{\bpair}[2]{{\big({#1},\,{#2}\big)}}

\newcommand{\at}[1]{{\left({#1}\right)}}
\newcommand{\nat}[1]{(#1)}
\newcommand{\bat}[1]{{\big(#1\big)}}

\newcommand{\ul}[1]{\underline{#1}}

\newcommand{\D}{\displaystyle}

\newcommand{\bigpar}{\par\quad\newline\noindent}

\newcommand{\abs}[1]{\left|{#1}\right|}
\newcommand{\babs}[1]{\big|{#1}\big|}

\newcommand{\nn}{\nonumber}

\newcommand{\dint}[1]{\,\mathrm{d}#1}

\newcommand{\La}{{\Lambda}}

\newcommand{\eps}{{\varepsilon}}

\newcommand{\la}{{\lambda}}

\theoremstyle{plain}
\newtheorem{theorem}             {Theorem}[]
\newtheorem{corollary}  [theorem]{Corollary}
\newtheorem{lemma}      [theorem]{Lemma}
\newtheorem{proposition}[theorem]{Proposition}
\theoremstyle{definition}
\newtheorem{definition} [theorem]{Definition}

\newtheorem*{remark*}{Remark}
\begin{document}

\begin{abstract}
    In a dissipative Fermi-Pasta-Ulam-Tsingou chain particles interact with their nearest neighbors through anharmonic potentials and linear dissipative forces. We prove the existence of front solutions connecting two different uniformly compressed (or stretched) states at $\pm \infty$ using an implicit function argument starting at a suitable continuum limit in the case of large damping. A detailed analysis allows us to show monotonicity of waves and to determine sharp exponential decay rates for a wide class of potentials including Hertzian potentials.
\end{abstract}

\maketitle


\section{Introduction}

We consider an infinite chain of particles described by the system
\begin{equation}
\label{fputld}
\ddot{q}_n = \Phi^\prime (q_{n+1}-q_n)- \Phi^\prime (q_{n}-q_{n-1}) + \gamma (\dot{q}_{n+1}-2\, \dot{q}_n+\dot{q}_{n-1}),
\quad n\in \mathbb{Z},
\end{equation}
where $q_n(t)$ denotes the displacement of the $n$th particle from a reference position and the overdot stands for the time derivative.
Particles interact with their nearest neighbors through a smooth anharmonic  convex potential $\Phi$
and linear dissipative forces with  a damping constant $\gamma \geq 0$.

An important classical example is given by the generalized Hertzian potential
\begin{equation}
\label{hpot}
\Phi_{\alpha}(r)=\frac{1}{\alpha+1}\, r_+^{\alpha +1},
\end{equation}
where we note $r_+ = \max{(r,0)}$ and assume $\alpha >1$.
The case $\alpha = 3/2$ is of particular interest because
the elastic contact force between two slightly compressed spherical beads
(or smooth non-conforming surfaces) scales like
$\Phi_{3/2}^\prime (r) = r_+^{3/2}$
when bead centers move closer by a distance $r$
(the vanishing of $\Phi_\alpha^\prime$ on the negative axis corresponding to contact loss).
Therefore, the system (\ref{fputld})-(\ref{hpot}) describes the collisional dynamics of a granular chain
in a quasi-static approximation \cite{nesterenko}.
The form of the linear dissipative term in (\ref{fputld})
is valid when the relative displacements
\begin{equation}
\label{defreldis}
r_n = q_{n+1}-q_n
\end{equation}
are nonnegative,
i.e. no contact breaking occurs.
Linear contact dissipation is a classical simplifying assumption (see e.g. \cite{lee-herrmann}), but
more realistic nonlinear dissipative forces can be considered
(see \cite{kuwabara-kono,luding-et-al,james} and references therein).

The case $\gamma =0$ corresponds to the Fermi-Pasta-Ulam-Tsingou (FPUT)
Hamiltonian lattice \cite{campbell-et-al}. Under quite general conditions on the anharmonic interaction potential $\Phi$,
this model supports solitary wave solutions, i.e. spatially localized traveling waves leaving the chain at rest at infinity.
These solutions take the form
\begin{equation}
\label{twdef}
r_n (t) =r_0(t-n/c),
\end{equation}
where $c \neq 0$ denotes the wave velocity and
$\lim_{t\rightarrow \pm\infty} r_n(t)=\lim_{n\rightarrow \pm\infty} r_n(t)=0$. Existence is shown e.g. in  \cite{FW94,Io00,Her10}.  Asymptotic descriptions and dynamic stability have been established for long-wave KdV type wave in \cite{FPI99,FPII02,FPIII04,FPIV04,Mizumachi2009} and for localized high-energy waves in \cite{FM02,HM15,HM19b}. For further qualitative properties see the review \cite{Vai22} and references therein. Additional results for Hertzian potentials
are described in  \cite{english-pego,stefanov-kevrekidis,james-pelinovsky}. Heteroclinic waves in some multi-well potentials have been studied in \cite{TV05,HZ09,HMSZ13}, where some asymptotic states are then typically periodic solutions.

The present work, for $\gamma >0$, deals with
another class of traveling wave solutions consisting of fronts that
connect two different uniformly compressed (or stretched) states at $\pm \infty$.
We assume more precisely
\begin{equation}
\label{rinfty}
\lim_{n\rightarrow \pm\infty} r_n(t)= r_{\pm}, \quad r_+ \neq r_- .
\end{equation}
A jump condition relating these asymptotic states and the wave velocity
was derived in a number of works
\cite{AP09,HR10,james}
and bears an analogy (through the hyperbolic continuum limit of (\ref{fputld}))
with the Rankine-Hugoniot conditions for hyperbolic systems
\cite{AP09,herrmann-rademacher10a}.
It takes the form
\begin{equation}
\label{v2}
\llbracket{r}\rrbracket\, c^2 = \llbracket{\Phi^\prime}\rrbracket
\end{equation}
with $\llbracket r  \rrbracket = r_{+\infty} - r_{-\infty}$, $\llbracket \Phi^\prime  \rrbracket = \Phi^\prime(r_{+\infty}) - \Phi^\prime(r_{-\infty})$,
so that $\llbracket{r}\rrbracket$ and $\llbracket \Phi^\prime  \rrbracket$ must have the same sign for traveling fronts to exist,
and the wave velocity does not depend on $\gamma$.

In the nondissipative case $\gamma=0$, an additional
necessary condition for the existence of traveling fronts
was derived in \cite{AP09,HR10,james} in terms of the coefficient
\begin{equation}
\label{consint2}
A =
\llbracket{\Phi}\rrbracket  - \llbracket r  \, \rrbracket \frac{\Phi^\prime(r_{-\infty})+\Phi^\prime(r_{+\infty})}{2}
\end{equation}
with $\llbracket \Phi  \rrbracket = \Phi(r_{+\infty}) - \Phi(r_{-\infty})$,
requiring $A=0$ for traveling fronts to exist.
The existence of traveling fronts in FPUT lattices
with appropriate interaction potentials satisfying $A=0$ was established in
\cite{Io00,HR10,H11}.
In contrast, $A$ does not vanish if $\Phi^{(3)}$ does not change sign, implying the
nonexistence of traveling fronts \cite{AP09,HR10,james}.
This nonexistence result applies in particular to the
generalized Hertz potential defined in (\ref{hpot}) in the compressive domain $r\geq 0$.

However, traveling fronts can be generated experimentally,
typically by compressing a granular chain with a piston at constant velocity; see
\cite[section 1.8.7]{nesterenko} and \cite{molinari-daraio}.
A correct modeling of this phenomenon requires to take contact damping into account, i.e.
to assume $\gamma >0$ in (\ref{fputld}).
Compression fronts in dissipative granular chains were studied in \cite{herbold-nesterenko,liang-et-al,james}
through numerical simulations and formal multiscale expansions, both for linear and nonlinear contact dissipation.
Similar continuum limit approximations were applied to other interaction potentials in \cite{duvall-et-al,hietarinta-et-al,AP09}.
These techniques permit in particular the analysis of the transition from oscillatory (underdamped) to monotonic (overdamped) fronts,
depending on the amount of dissipation induced by the model parameters and boundary conditions.
Moreover, it was established in \cite{AP09,james} that the velocity of dissipative fronts has the sign of
$A$, which must be nonzero for dissipative fronts to exist.

In this paper, we present the first rigorous study of traveling waves for damped FPUT chains. We will use an implicit function argument starting at a suitable continuum limit in the case of very large damping. The main technical difficulties arise when showing continuity of the relevant derivatives for a setting that includes Hertzian potentials \eqref{hpot}. We will use tools for exponentially weighted, fractional Sobolev spaces including the Kato-Ponce inequality.  Beyond mere existence we provide quantitative and qualitative properties by showing monotonicity of the waves and sharp exponential decay rates. Monotonicity ensures that the wave remains in the non-flat part of the Hertzian potential \eqref{hpot}. Monotonicity for all potentials will be shown via sharp lower exponential estimates for the tail behavior. The implicit function argument will already use function spaces with exponential decays. We identify two different auxiliary continuum limits for $x\to -\infty$ and for $x\to +\infty$, these will yield the correct decay rates by a careful analysis of their Fourier symbols. Solutions can be shown to have these exact rates by applying a conditional improvement lemma for a finite number of times.

\subsection*{Outline of the paper:}
In Section 2, we formulate the traveling wave equation and express it using convolution operators. The statement of the main results is given for general potentials with regularity assumptions that are also satisfied in the Hertzian case. In Section 3 we collect preliminaries about convolution operators and their Fourier symbols. Section 4 includes the implicit function argument for existence. In Section 5 we study an auxiliary problem for the spatial derivative of the waves, this provides a proof of monotonicity and the sharp exponential rates at $\pm \infty$. Section 6 contains some postponed proofs of earlier results on  asymptotic  estimates of the Fourier symbols.

\subsection*{Acknowledgments}

The authors would like to thank Bj\"orn de Rijk and Guido Schneider for organizing the workshop \lq Spatial Dynamics and related approaches'  at the University of Stuttgart in September 2022, where this work was initiated.  
\section{Problem formulation}

\subsection{Nonlinear integral equation for traveling waves}

Based on equation (\ref{fputld}),
the relative displacements $r(t)=(r_n(t))_{n \in \mathbb{Z}}$ defined in (\ref{defreldis}) satisfy
\begin{equation}
\label{eqevolr}
\ddot{r}=\Delta{\Phi^\prime (r)}+\gamma \Delta{\dot{r}},
\end{equation}
where $\Delta$ refers to the  discrete Laplacian
$(\Delta{r})_n = r_{n+1}-2\, r_n +r_{n-1}$.

In the sequel, we look for traveling wave solutions of (\ref{eqevolr}) with velocity $c$.
We consider a regime of high dissipation $\gamma \gg 1$, which leads us to introduce the
small parameter $\eps =1/\gamma$. Moreover we assume a slowly varying traveling wave
profile by setting
\begin{equation}
\label{twansatz}
r_n (t) = R(x), \quad x=\eps\, (n-c\, t).
\end{equation}
Substituting (\ref{twansatz}) in equation (\ref{eqevolr}) yields the
nonlinear advance-delay differential equation
\begin{align}
\label{TWEqn}
c^2\,\tfrac{\dint^2}{\dint x^2} R=\Delta_\eps\at{\Phi^\prime\at{R}-c\,\tfrac{\dint}{\dint x} R }\, ,
\end{align}
where $0<\eps \ll1$ and
$\Delta_\eps$ denotes the discrete Laplacian with spacing $\eps$ (and prefactor $1/\eps^2$)
$$
[\Delta_\eps {R} ](x)= \frac{1}{\eps^2}\, \big( R(x+\eps )-2\, R(x) + R(x-\eps ) \big).
$$
For a front solution,
boundary conditions read
\begin{equation}
\label{bcfront}
\lim_{x\rightarrow \pm\infty} R(x)= r_{\pm}.
\end{equation}
We assume without loss of generality
\begin{equation}
\label{bchyp}
r_- > r_+
\end{equation}
(if $r_- < r_+$, the above case can be recovered by the change $x \rightarrow -x$ and $c \rightarrow -c$ in (\ref{TWEqn})).
Equation \eqref{TWEqn} can be integrated twice by using the identity
$$
\Delta_\eps {R} = \tfrac{\dint^2}{\dint x^2} (\La_\eps \ast R),
$$
where $\ast$ denotes convolution product and
$\La_\eps$ is the tent map with width $2\eps$ and height $1/\eps$
$$
\La_\eps (x)= \frac{1}{\eps}\, \La_1 \left(\frac{x}{\eps}\right), \quad
\La_1 (x)= \mathrm{max}(1-|x|,0).
$$
This results in the nonlinear integral equation
$$
d+c^2\,R+c\,\La_\eps^\prime\ast R = \La_\eps\ast \Phi^\prime\at{R}\,,
$$
where $d$ is a constant of integration
(no linear term in $x$ occurs since $R(x)$ is bounded on $\mathbb{R}$),
or equivalently
\begin{equation}
\label{inteqtw}
d+c^2\,R+c\,\La_\eps\ast R^\prime = \La_\eps\ast \Phi^\prime\at{R}\, .
\end{equation}
The constants $c,d$ in \eqref{inteqtw} are determined by the boundary conditions \eqref{bcfront}.
Indeed, since $\La_\eps \in \fspaceL^1 (\mathbb{R})$ with $\int_{\mathbb{R}}{\La_\eps\, dx}=1$,
it follows by dominated convergence that for all
$f\in  \fspaceC^0 (\mathbb{R})$ satisfying $\lim_{x\rightarrow\pm\infty}{f(x)}=\ell_{\pm}$ one has
$$
\lim_{x \rightarrow\pm\infty}{(\La_\eps\ast f )}=\ell_{\pm}.
$$
Consequently, for all $R \in  \fspaceC^1 (\mathbb{R})$
satisfying \eqref{bcfront}-\eqref{inteqtw} and $\lim_{x\rightarrow\pm\infty}{R^\prime (x)}=0$, one obtains
by letting $x\rightarrow \pm\infty$ in \eqref{inteqtw}
\begin{equation}
\label{vald}
d   = \Phi^\prime\at{r_{\pm}} - c^2\, r_{\pm}
\end{equation}
and the jump condition \eqref{v2} follows.

Let us now assume $\Phi^\prime \in \fspaceC^1([r_+ , r_-])$ and
\begin{equation}
\label{hypphiprime}
\Phi^\prime
\mbox{ increasing and strictly convex on~} [r_+ , r_-] .
\end{equation}
Assuming $\Phi^\prime$ increasing implies $\llbracket{\Phi^\prime}\rrbracket  / \llbracket{r}\rrbracket >0$ in \eqref{v2},
which is necessary for the existence of traveling fronts.
Moreover, since $\Phi^\prime$ is strictly convex,
its graph lies below the chord through $r_+$ and $r_-$ and thus
the coefficient $A$ in \eqref{consint2} is positive
($A$ is the signed area between the chord and the graph of $\Phi^\prime$).
Under this condition,
it was proved in \cite{AP09,james} that $c>0$ and then it follows from \eqref{v2} that
\begin{equation}
\label{valc}
c = \left( \llbracket{\Phi^\prime}\rrbracket  / \llbracket{r}\rrbracket \right)^{1/2}.
\end{equation}

\subsection{Continuum limit}

Our aim is to analyze front solutions to \eqref{inteqtw} for $\eps \approx 0$,
hence we start by considering the limit $\eps \rightarrow 0$ in \eqref{inteqtw}.
We note that the convolution kernel $\La_\eps$ is an approximation to the identity,
that is, $\lim\limits_{\eps\rightarrow 0}{\La_\eps}=\delta_0$, where
$\delta_0$ denotes the Dirac distribution at $x=0$.
Consequently, the formal continuum limit of \eqref{inteqtw} as $\eps \rightarrow 0$
is the first order ODE
\begin{equation}
\label{contlim}
c\,R^\prime = \Phi^\prime\at{R}- c^2\, R - d
\, .
\end{equation}
This problem corresponds to a continuum limit of \eqref{eqevolr}
for front solutions in the large dissipation limit.

Under assumption \eqref{hypphiprime},
for any choice of $r_->r_+$ there exists a unique choice for $d$, $c$ and a heteroclinic solution $R_0$ to
\eqref{contlim} (unique up to phase shift)
which is strictly decreasing and satisfies \eqref{bcfront}.
Indeed, $r_{\pm}$ are equilibria of \eqref{contlim} if and only if $c^2$ is given by
\eqref{v2} (where
$\llbracket{\Phi^\prime}\rrbracket  / \llbracket{r}\rrbracket >0$ since ${\Phi^\prime}$ is increasing)
and $d=\Phi^\prime\at{r_\pm}- c^2\, r_\pm$.
Then the strict convexity of $\Phi^\prime$ implies that the right hand side of \eqref{contlim} is negative on $(r_+ , r_-)$,
hence there exists a heteroclinic solution satisfying \eqref{bcfront}-\eqref{bchyp} provided $c>0$ is given by \eqref{valc}.
Consequently, the values of $c,d$ are identical for the integral equation \eqref{inteqtw}
and the continuum limit problem \eqref{contlim}.

\subsection{Renormalization}

Consider a solution $r(t)=(r_n(t))_{n \in \mathbb{Z}}$
of \eqref{eqevolr} with boundary conditions \eqref{rinfty}.
We renormalize this solution by defining
\begin{equation}
\label{normsol}
\tilde{r}_n (\tilde{t})=\frac{1}{\llbracket{r}\rrbracket}\, \left(  r_+ - r_n \left( \frac{\tilde{t}}{c} \right)  \right) ,
\quad  c=\left( \frac{\llbracket{\Phi^\prime}\rrbracket }{  \llbracket{r}\rrbracket }\right)^{1/2},
\end{equation}
so that
\begin{equation}
\label{bcnorm}
\lim_{n\rightarrow -\infty} \tilde{r}_n(\tilde{t})= 1, \quad
\lim_{n\rightarrow +\infty} \tilde{r}_n(\tilde{t})= 0.
\end{equation}
Similarly, we introduce the renormalized potential
$\tilde\Phi \in \fspaceC^2([0 , 1])$ defined by
\begin{equation}
\label{normpot}
\tilde{\Phi}(\tilde{r})=\frac{1}{\llbracket{r}\rrbracket \, \llbracket{\Phi^\prime}\rrbracket} \, {\Phi}(r_+ - \llbracket{r}\rrbracket\, \tilde{r})
+\frac{\Phi^\prime (r_+)}{ \llbracket{\Phi^\prime}\rrbracket }\, \tilde{r}.
\end{equation}
Condition \eqref{hypphiprime} for $\Phi^\prime$
on $[r_+,r_-]$ is equivalent to
the same condition on $[0,1]$ for $\tilde{\Phi}^\prime$, and one has in addition
$\tilde{\Phi}^\prime(0)=0$, $\tilde{\Phi}^\prime(1)=1$.
It follows that equation \eqref{eqevolr} is equivalent to the following problem for
$\tilde{r}(t)=(\tilde{r}_n(t))_{n \in \mathbb{Z}}$
\begin{equation}
\label{eqevolnorm}
\ddot{\tilde{r}}=\Delta{\tilde{\Phi}^\prime (\tilde{r})}+\tilde\gamma \Delta{\dot{\tilde{r}}},
\quad
\tilde\gamma = \frac{\gamma}{c}.
\end{equation}
From definition \eqref{normsol} it follows that $r$ is a traveling wave solution to \eqref{eqevolr} with velocity $c$
(i.e. taking the form \eqref{twdef}) if and only if
$\tilde{r}$ is a traveling wave solution to \eqref{eqevolnorm} with unit velocity.

The above renormalization shows that one can assume without loss of generality
$r_-=1$, $r_+ =0$, ${\Phi}^\prime(0)=0$, ${\Phi}^\prime(1)=1$, leading to
$c=1$ and $d=0$ (using \eqref{vald}) in \eqref{inteqtw}.

\subsection{Statement of the main results}

\begin{theorem}
\label{existthm}
Let $\Phi \in \fspaceC^{2,\beta}([0 , 1])$ such that $\beta \in (0,1) $,
$\Phi^\prime$ is increasing and strictly convex on $[0 , 1]$,
with ${\Phi}^\prime(0)=0$, ${\Phi}^\prime(1)=1$.
Consider the following problem for $\eps >0$
\begin{equation}
\label{norminteqtw}
\La_\eps\ast R^\prime + R= \La_\eps\ast \Phi^\prime\at{R}\, ,
\end{equation}
and its formal limit for $\eps \rightarrow  0$
\begin{equation}
\label{normcontlim}
R^\prime +R = \Phi^\prime\at{R}
\, ,
\end{equation}
both equipped with the boundary conditions
\begin{equation}
\label{bctwnorm}
\lim_{x\rightarrow -\infty} R(x)= 1, \quad
\lim_{x\rightarrow +\infty} R(x)= 0
\end{equation}
and the phase condition
\begin{equation}
\label{phase}
R(0)=\frac{1}{2}.
\end{equation}
Then there exists a unique solution $R_0$ to \eqref{normcontlim}-\eqref{bctwnorm}-\eqref{phase}.
Moreover, there exists $\eta >0$ and $\eps_1 >0$ such that
for all $\eps \in (0, \eps_1]$, problem
\eqref{norminteqtw}-\eqref{bctwnorm}-\eqref{phase}
admits a unique solution $R_\eps$ such that $\| R_\eps - R_0  \|_{\fspaceH^1(\mathbb{R})} \leq \eta$. Moreover, 
\begin{equation} \label{H1est}
    R_\eps = R_0  +O\at{\eps} \mbox{ in } \fspaceH^1(\mathbb{R}) \mbox{ as } \eps \rightarrow 0. 
\end{equation}
Furthermore $R_\eps$ is monotone in $x$ and converges exponentially for $x \to \pm \infty$.  
\end{theorem}

\begin{remark*}

\begin{enumerate}
    \item  We note that the Hertz potentials in \eqref{hpot} with $\alpha \in (1,2)$ satisfy the assumptions with H\"older exponent $\beta= \alpha-1$. Smooth potentials $\Phi \in \fspaceC^3 ([0 , 1])$ with ${\Phi}^\prime(0)=0$, ${\Phi}^\prime(1)=1$ satisfy the assumptions if 
$\Phi^{\prime\prime}$ and $\Phi^{\prime\prime\prime}$ are positive   in $[0 , 1]$.   
\item Formulas for the exponential rates for $x \to \pm \infty$ are given in Proposition \ref{Lem.AP.ModKernels}, see also Theorem \ref{thm:rates}.
\end{enumerate}  
\end{remark*}

\section{Preliminaries}
\subsection{Setting}
In a chain with linear dissipation, the rescaled equation for the distance profile of a
traveling wave takes the form \eqref{norminteqtw} with
$0<\eps=1/\gamma\ll1$.

In the sequel, we use the following variant of the Fourier transform
\begin{align*}
\hat{a}\at{k} =\int\limits_{-\infty}^{+\infty}\exp\at{-2\,\pi\,\iu\,k\,x}\, a\at{x}\dint{x}\,,\qquad a\at{x}=
\int\limits_{-\infty}^{+\infty}\exp\at{+2\,\pi\,\iu\,k\,x}\, \hat{a}\at{k}\dint{k}.
\end{align*}
%
%
%
%
%
%
%
%
Applying the Fourier transform to
problem \eqref{norminteqtw} yields
\begin{align*}
    \widehat{\La_\eps}(k) \widehat{R'}(k) + \hat{R}(k)= \widehat{\La_\eps}(k) \widehat{\Phi^\prime\at{R}} (k)
\end{align*}
with $\widehat{\La_\eps}(k)=\sinc^2\at{\eps\,\pi\,k} $ and  $\widehat{R'}(k)=2\,\pi\,\iu\,k \widehat{R}(k)$, such that  \eqref{norminteqtw} can be written as the fixed point problem
\begin{align}
R=a_\eps\ast\Phi^\prime\at{R}\,,  \label{eq:fixedpt}
\end{align}
where
\begin{align*}
\hat{a}_\eps\at{k} = \frac{\sinc^2\at{\eps\,\pi\,k}}{1+2\,\pi\,\iu\,k\,\sinc^2\at{\eps\,\pi\,k}}\,.
\end{align*}
The pointwise limit of $\widehat{a_\eps}$ yields the operator
\begin{align*}
\hat{a}_0\at{k}=\frac{1}{1+2\,\pi\,\iu\,k\,},
\quad\qquad
a_0\at{x}=
\left\{
\begin{array}{ccc}
\exp\at{-x}&&\text{for $x\geq0$,}\\
0&&\text{for $x<0$.}
\end{array}
\right.
\end{align*}
Then the limiting equation of  \eqref{eq:fixedpt} is equivalent to 
\begin{align}\label{eqn:R0}
  R_0+ R_0' & = \Phi'(R_0).
\end{align}
Imposing the initial condition $R_0(0)= \frac{1}{2}$ this scalar ordinary differential equation has a unique solution. The solution satisfies
\begin{align}\label{eqn:R0bc}
  R_0(-\infty)=1, \quad R_0(\infty)=0
\end{align}
as $\Phi^\prime$ is increasing and strictly convex on $[0 , 1]$,
with ${\Phi}^\prime(0)=0$, ${\Phi}^\prime(1)=1$. 

A major part of the analysis is to provide quantitative estimates for $a_\eps$ and 
for some modified kernels. The linearisation of the force will be relevant for the further analysis. We denote 
\begin{align}
\label{eqn:Lin0}
P=\Phi^{\prime\prime}\at{R_0}\,.
\end{align}
In particular, $P$ is nonnegative and satisfies
\begin{align}
\label{AsympStatesP}
\lim_{x\to\pm\infty} P\at{x}=p_\pm\,,\qquad
p_\pm:=\Phi^{\prime\prime}\at{r_\pm}.
\end{align}
\bigpar\emph{Remark} In the Hertzian case as in \eqref{hpot} we have $p_+=0$.

\subsection{Properties of the convolution operator}
The main properties of the convolution operator follow from the next proposition about the Fourier transformed kernel $\hat{a}_\eps$.

\begin{proposition}\label{prop:symbol}
Let $\eta_-, \eta_+$ be such that $0>-\eta_-> 1-p_-$  and $0< \eta_+< 1-p_+$. Also let $0<s<1$. Then there exists $\eps_0>0$ and $C$ such that for  all $0<\eps\leq \eps_0$  with  $\Breve{\eta}_+:= \frac{1}{2}(\eta_++(1-p_+))$ and 
$\Breve{\eta}_-:= \frac{1}{2}((p_- -1)+\eta_-)$  the Fourier symbol $\hat{a}_\eps$ is holomorphic in $ \Rset \times \iu [-\Breve{\eta}_-,\Breve{\eta}_+]$  and the following uniform estimates hold
\begin{align}
    &\hat{a}_\eps(.) (1+ |.|)\in \fspaceL^\infty(\Rset \times \iu [-\Breve{\eta}_-,\Breve{\eta}_+]), \label{eqn:fourlinftyest}\\
    &\|\hat{a}_\eps- \hat{a}_0 \|_{\fspaceL^\infty(\Rset \times \iu [-\Breve{\eta}_-,\Breve{\eta}_+]) }\leq C \eps,
    \label{eqn:fourerror}\\
    &\|(\hat{a}_\eps(.)- \hat{a}_0(.)) (1+ |.|^{1-s}) \|_{\fspaceL^\infty(\Rset \times \iu [-\Breve{\eta}_-,\Breve{\eta}_+]) }\leq C \eps^{s},
    \label{eqn:fourerror2}\\
    &\sup_{\eta \in [-\Breve{\eta}_-,\Breve{\eta}_+]} \int_\Rset \left|\hat{a}_\eps\left(k + \iu \frac{\eta}{2 \pi} \right)\right|^2\dint{k} < C. \label{eqn:fourL2exp}
\end{align}
\end{proposition}
The proof will be given in section \ref{Sec:Four}. From Proposition \ref{prop:symbol} we  obtain several exponential estimates.
\begin{corollary} \label{cor:a} Let $\eta_-,\Breve{\eta}_-, \eta_+,\Breve{\eta}_+ $ be as in Proposition \ref{prop:symbol} then 
\begin{align}
  \label{eqn:L2exp}  \int_{\Rset} a_\eps^2 (x) \exp ( 2\eta x) \dint{x}&< C \mbox{ for }  \eta \in [-\Breve{\eta}_-, \Breve{\eta}_+],\\
   \label{eqn:posL1exp}  \int_0^\infty a_\eps(x) \exp ( \eta x) \dint{x}& < \frac{C}{\sqrt{\Breve{\eta}_+-\eta}} \mbox{ for } \eta \in [0,\Breve{\eta}_+), \\ \label{eqn:negL1exp} 
    \int^0_{-\infty} a_\eps(x) \exp ( \eta x) \dint{x}& < \frac{C}{\sqrt{ \Breve{\eta}_-+ \eta}} \mbox{ for } \eta \in (-\Breve{\eta}_-,0].
\end{align}
\end{corollary}
\begin{proof}
    The estimate \eqref{eqn:L2exp} follows with \cite[Thm IX.13]{RSbook} and \eqref{eqn:fourL2exp}. Then with the Cauchy-Schwarz inequality  we obtain \eqref{eqn:posL1exp}
\begin{align*}
    \int_0^\infty a_\eps(x) \exp ( \eta x) \dint{x}&= 
    \int_0^\infty a_\eps(x) \exp ( \Breve{\eta}_+ x) 
    \exp ( (\eta-\Breve{\eta}_+) x) \dint{x} \\&\leq
    \| a_\eps(.) \exp( \Breve{\eta}_+ .) \|_{\fspaceL^2} 
    \sqrt{ \int_0^\infty  \exp ( 2 (\eta-\Breve{\eta}_+) x) \dint{x} }
    \leq \frac{C}{\sqrt{\Breve{\eta}_+-\eta}}
\end{align*}
    and similarly \eqref{eqn:negL1exp}.
\end{proof}
 We now introduce suitable exponentially weighted function spaces. 
\begin{definition} For $\eta \in \Rset$, define 
  \begin{align*}
      \fspaceH_\eta^1(\Rset)&= \{ w \mid \exp( \eta .) w(.) \in  \fspaceH^1(\Rset) \} & \mbox{  with norm } \|w\|_{\fspaceH_\eta^1}=& \|\exp( \eta .) w(.)\|_{\fspaceH^1},\\
      \fspaceL_\eta^2(\Rset)&= \{ w \mid \exp( \eta .) w(.) \in  \fspaceL^2(\Rset) \} &\mbox{  with norm } \|w\|_{\fspaceL^2_\eta}=& \|\exp( \eta .) w(.)\|_{\fspaceL^2}.
  \end{align*}
\end{definition}
With $\eta \in (-\Breve{\eta}_-, \Breve{\eta}_+) $,  $w \in \fspaceH_\eta^1(\Rset) $ and $v= a_\eps \ast w$ we define
\begin{align}
     \tilde{w}(.)&= \exp( \eta .) w(.), \nonumber\\
     \tilde{v}(.)&= \exp( \eta .) v(.)= \exp( \eta .) 
     a_\eps \ast (\exp( -\eta .) \tilde{w}(.)) = \tilde{a}_\eps \ast \tilde{w}, \label{eqn:v}
\end{align} 
where transformed kernel is given by
\begin{align}\label{def:tildea}
    \tilde{a}_\eps(x)= a_\eps(x) \exp(\eta x). 
\end{align}
We also obtain 
\begin{align}\tilde{v}'&=
\tilde{a}_\eps \ast \tilde{w}'. \label{eqn:v'} 
\end{align} 
The transformed kernel satisfies $\tilde{a}_\eps  \in \fspaceL^1(\Rset)$ by Corollary \ref{cor:a}, and its Fourier transform can be calculated as
\begin{align}
    \label{eqn:tildeaFour}
    \hat{\tilde{a}}_\eps(k)= \hat{a}_\eps\left(k-\frac{\eta}{2 \pi \iu}\right)=\hat{a}_\eps\left(k+\frac{\iu \eta}{2 \pi}\right).
\end{align}
The convolution operator is continuous on the exponentially weighted spaces as shown below.
\begin{proposition}\label{prop:weight}
   \begin{enumerate}
       \item For $ w \in \fspaceL^2_\eta(\Rset)$ the map  $w \mapsto 
       a_\eps \ast w \in \fspaceH^1_\eta(\Rset) $ is well-defined and continuous  with respect to $w$, uniformly for $\eta \in (-\Breve{\eta}_-, \Breve{\eta}_+)$ and for $0< \eps< \eps_0$.
       \item The operators are continuous with respect to $\eps$, i.e. there exists $C>0$ such that
       \begin{align}
           \label{eqn:expaeps}
           \| a_\eps \ast w - a_0 \ast w \|_{\fspaceH^1_\eta(\Rset)}
           & \leq C \eps \|w \|_{\fspaceH^1_\eta(\Rset)} 
                  \end{align}
       uniformly for $\eta \in (-\Breve{\eta}_-, \Breve{\eta}_+)$.
       \item Furthermore for $0<s<1$, there exists $C_s>0$ such that
       \begin{align}
            \| a_\eps \ast w - a_0 \ast w \|_{\fspaceH^1_\eta(\Rset)}
           & \leq C \eps^s \|w \|_{\fspaceH^s_\eta(\Rset)} \label{eqn:expaeps2} 
       \end{align}
       uniformly for $\eta \in (-\Breve{\eta}_-, \Breve{\eta}_+)$.
   \end{enumerate}  
\end{proposition}
\begin{proof}
This follows from Proposition \ref{prop:symbol}.    For  a  fixed $\eta$ we have with \eqref{eqn:v} and \eqref{eqn:v'} using the Plancherel theorem
    \begin{align*}
        \|a_\eps \ast w \|_{\fspaceH^1_\eta(\Rset)} =
        \|\tilde{a}_\eps \ast \tilde{w} \|_{\fspaceH^1(\Rset)}
        \leq C \sup_{k \in \Rset} \left| \hat{a}_\eps\left(k+\frac{\iu \eta}{2 \pi}\right) (1+|k|) \right| \| \tilde{w} \|_{\fspaceL^2(\Rset)},
    \end{align*}
    then part (1) follows from \eqref{eqn:fourlinftyest}. 
    We also obtain
     \begin{align*}
        \|a_\eps \ast w -a_0 \ast w  \|_{\fspaceH^1_\eta(\Rset)} =&
        \|(\tilde{a}_\eps -\tilde{a}_0) \ast \tilde{w} \|_{\fspaceH^1(\Rset)}\\
        \leq &\sup_{k \in \Rset} \left| \hat{a}_\eps\left(k+\frac{\iu \eta}{2 \pi}\right) - \hat{a}_0\left(k+\frac{\iu \eta}{2 \pi}\right)\right| \| \tilde{w} \|_{\fspaceH^1(\Rset)},
    \end{align*}    
    such that part (2) follows with \eqref{eqn:fourerror}. And 
    combining both together and the boundedness of the symbol  $\hat{a}_0$
    with \eqref{eqn:fourerror2} yields:
\begin{align*}
        \|a_\eps \ast w -a_0 \ast w  \|_{\fspaceH^1_\eta(\Rset)} &=
        \|(\tilde{a}_\eps -\tilde{a}_0) \ast \tilde{w} \|_{\fspaceH^1(\Rset)}\\
        &\leq \sup_{k \in \Rset} \left| (1+|k|^{1-s}) \left(\hat{a}_\eps\left(k+\frac{\iu \eta}{2 \pi}\right) - \hat{a}_0\left(k+\frac{\iu \eta}{2 \pi}\right)\right)\right| \| \tilde{w} \|_{\fspaceH^{s}(\Rset)}\\ 
        &\leq \Bigl[ (1 + \frac{1}{\eps^{1-s}}) \sup_{|k|<1/\eps } \left|  \left(\hat{a}_\eps\left(k+\frac{\iu \eta}{2 \pi}\right) - \hat{a}_0\left(k+\frac{\iu \eta}{2 \pi}\right)\right)  \right| \\& \quad +  2 \eps^s \sup_{|k|\geq 1/\eps } \left|  |k| \left(\hat{a}_\eps\left(k+\frac{\iu \eta}{2 \pi}\right) - \hat{a}_0\left(k+\frac{\iu \eta}{2 \pi}\right)\right)  \right|  \Bigr]  \| \tilde{w} \|_{\fspaceH^{s}(\Rset)}\\ &\leq C \eps^s \| \tilde{w} \|_{\fspaceH^{s}(\Rset)},
    \end{align*}
which completes the proof of part (3).
\end{proof}

\section{Existence of fronts using an implicit function argument}\label{sec:IFT}

Using the estimates in the previous sections we check the assumptions in an implicit function argument. For that, 
we first rewrite \eqref{eq:fixedpt} using $R_0$ as in \eqref{eqn:R0} and let 
\begin{align} 
  F(\eps,W) = \begin{cases}
      R_0+W- a_\eps \ast \Phi^\prime\at{R_0+W} & \text{ for } \eps \geq 0,  \\F(- \eps,W) &\text{ for } \eps <0.
  \end{cases} \label{eq:F}
\end{align}
It is also convenient to extend $\Phi^\prime$
outside $[0,1]$ such that the second derivative $\Phi^{\prime\prime}$ is constant on 
$\mathbb{R}_0^-$ and 
on $[1,+\infty)$.
Then suitable zeros of $F$ are the desired traveling waves.
We use the following version of the implicit function theorem   as in \cite[XVII.4.1]{KantFunc}.
\begin{theorem}\label{thm:IFT}
  Let $X,Y,Z$ be Banach spaces, let $F$ be an operator defined in a neighborhood $\Omega$ of a point $(x_0,y_0) \in X \times Y$ and mapping $\Omega$ to $Z$. Suppose
  \begin{enumerate}
    \item $F$ is continuous at $(x_0,y_0)$;
    \item $F(x_0,y_0)=0$;
    \item $D_2 F$ exists in $\Omega$ and is continuous at $(x_0,y_0)$;
    \item the operator $D_2 F(x_0,y_0) \in B(Y,Z)$ has a continuous inverse in $B(Z,Y)$.
  \end{enumerate}
  Then there exists an operator $G$ defined in a neighborhood $U \subset X$ of $x_0$ and mapping $U$ into $Y$ with the following properties:
  \begin{enumerate}
    \item[(a)] $F(x,G(x))=0$ for $x \in U$;
    \item[(b)] $G(x_0)=y_0$;
    \item[(c)] $G$ is continuous at $x_0$.
  \end{enumerate}
  The operator $G$ defined by (a)-(c) is unique in a neighborhood of $x_0$.
\end{theorem}
For $\eta_-, \eta_+$ as in Proposition \ref{prop:symbol} we set
\begin{align*}
  X  := \Rset, \quad
  Y  := \{w \in \fspaceH_{\eta_+}^1(\Rset) \cap \fspaceH_{-\eta_-}^1(\Rset)\mid w(0)=0 \},  \quad
  Z  :=  \fspaceH_{\eta_+}^1(\Rset) \cap \fspaceH_{-\eta_-}^1(\Rset).
\end{align*}
As an auxiliary space we set 
\[ Z_0  :=  \fspaceL_{\eta_+}^2(\Rset) \cap \fspaceL_{-\eta_-}^2(\Rset).\]
We fix a finite ball in $X \times Y$ around $(0,0)$ and name it $\Omega$. We can now check that $F$ satisfies the assumptions of Theorem \ref{thm:IFT}. First we rewrite \eqref{eq:F}
\begin{align}\label{eq:Fsep}
  F (\eps,W) = F_1(\eps)&+F_2(\eps,W) \text{ with }\\
 F_1(\eps)&:= R_0- a_\eps \ast \Phi^\prime\at{R_0}  \nonumber \\
F_2(\eps,W)&:=  W - \left(a_\eps \ast \Phi^\prime\at{R_0+W} - a_\eps \ast \Phi^\prime\at{R_0} \right) .\nonumber
\end{align}
\subsection{Continuity properties}
\begin{lemma}\label{lem:F1}
 The constant part $F_1$ of the operator $F$ satisfies   $F_1: X \to Z $  and is continuous at $\eps=0$.
\end{lemma}
\begin{proof}
We rewrite 
    \begin{align*}
        F_1(\eps)&:= R_0- a_\eps \ast \Phi^\prime\at{R_0}= R_0- a_\eps \ast R_0 - a_\eps \ast \left( \Phi^\prime\at{R_0} - R_0 \right).
    \end{align*}
For $\eta \in (0,\eta_+]$, we have due the ODE \eqref{eqn:R0} that $\Phi^\prime\at{R_0} - R_0 \in \fspaceH_{\eta}^1(\Rset) $. Then $a_\eps \ast \left( \Phi^\prime\at{R_0} - R_0 \right) \in \fspaceH_{\eta}^1(\Rset)$ by Proposition \ref{prop:weight}, which also implies continuity at $\eps=0$. For $\eta \in (-\eta_-,0)$, we observe $R_0 \in\fspaceH_{\eta}^1(\Rset)$ such that we also have  $a_\eps \ast R_0 \to a_0 \ast R_0$ in  $\fspaceH_{\eta}^1(\Rset)$ by Proposition \ref{prop:weight}.

For $\eta \in [-\eta_-,0)$, we again have due to the ODE \eqref{eqn:R0} $\Phi^\prime\at{R_0} - R_0 \in \fspaceH_{\eta}^1(\Rset) $ such that  $a_\eps \ast \left( \Phi^\prime\at{R_0} - R_0 \right) \to 
a_0 \ast \left( \Phi^\prime\at{R_0} - R_0 \right)$ in  $\fspaceH_{\eta}^1(\Rset)$ by Proposition \ref{prop:weight}.
For the other part we observe
\begin{align*}
     R_0- a_\eps \ast R_0=  (R_0-1) - a_\eps \ast (R_0-1).
\end{align*}
We note that $R_0-1  \in\fspaceH_{\eta}^1(\Rset)$ which implies  $a_\eps \ast (R_0 -1) \to a_0 \ast (R_0-1) $ 
in  $\fspaceH_{\eta}^1(\Rset)$ as required. \end{proof}
\begin{lemma} \label{lem:superpos}
    The map
    \begin{align*}
        W \mapsto U:= (\Phi^\prime\at{R_0+W} - \Phi^\prime\at{R_0})
    \end{align*}
   is continuous as map $Y \to Z_0$. 
\end{lemma}
\begin{proof}
For any exponential weight $\eta \in [-\eta_-,\eta_+]$ we let 
\begin{align*}
    W= \exp( -\eta \cdot ) \tilde{W}, \qquad   U= \exp( -\eta \cdot ) \tilde{U}.
\end{align*}
Then
\begin{align*}
    \tilde{U}&= \exp( \eta \cdot ) \left[ \Phi^\prime\at{R_0+W} - \Phi^\prime\at{R_0} \right] .  
\end{align*}
Noting that $R_0 \in \fspaceL^\infty (\mathbb{R})$ and $\|W\|_{\fspaceL^\infty}\leq C \|W\|_Y $ and using $\Phi \in \fspaceC^2$  we have that
$\Phi^\prime$ 
is uniformly Lipschitz on $\Omega$ as the range of  $R_0(.)+W(
.)$ with $W \in \Omega$ is contained in a fixed interval. Consequently  
\begin{align*}
    |\Phi^\prime\at{R_0+W} - \Phi^\prime\at{R_0}| &\leq C \tilde{W} \exp(- \eta \cdot),
\end{align*}
which yields  $U \in Z_0$ and continuity with respect to $W$. 
\end{proof}
\begin{lemma} \label{lem:F2}
    The non-constant part $F_2$ of the operator $F$ satisfies   $F_2: \Omega  \to Z $. It is continuous with respect to $\eps$ at $\eps=0$ and it is continuous with respect to $W$  in $\Omega$.
\end{lemma}
\begin{proof} 
 Continuity with respect to $W$ follows from the previous lemma \ref{lem:superpos} and Proposition \ref{prop:weight}, part (1). Then continuity with 
 respect to $\eps$ follows by Proposition \ref{prop:weight} part (2).
\end{proof}
\begin{lemma} \label{lem:F=0}
$F(0,0)=0.$
\end{lemma}
\begin{proof}
For $\eps=0$, we have
\begin{align*}
F(0,0) & = F_1(0)=  R_0- a_0 \ast \Phi^\prime\at{R_0} =0  \end{align*}
as 
 \begin{align*}  R_0 & = a_0 \ast \Phi^\prime\at{R_0} \end{align*}
is equivalent to
 \begin{align*}   R_0+ R_0^\prime & =\Phi^\prime\at{R_0},\end{align*}
 which is satisfied by $R_0$ as in \eqref{eqn:R0}. 
\end{proof}

\subsection{Differentiability}

\begin{lemma} \label{lem:FisFrech}
The function  $F: \Omega \rightarrow Z$ is   Fr\'echet differentiable with respect to $W$. 
The derivative  given by
    \begin{align*}
            D_2 F(\eps,W): Y & \rightarrow Z \,, \\
         D_2 F(\eps,W) V&= V - a_\eps \ast (\Phi^{\prime\prime}\at{R_0+W} V)\,,
          \end{align*}
      is  continuous at $(\eps,W)=(0,0)$.
\end{lemma}
\begin{proof}
We first study
the Fr\'echet differentiability of 
\begin{align*}
    W \mapsto H(W)=  \Phi^\prime\at{R_0+W}-\Phi^\prime\at{R_0}
\end{align*}
as a  map $\Omega \to Z_0$. 
We aim to show that its derivative is given by 
\begin{align*}D H(W) V= \Phi^{\prime\prime}\at{R_0+W} V.\end{align*} Hence
we need to prove that
\begin{align*}
    U:= \Phi^\prime\at{R_0+W+V}-\Phi^\prime\at{R_0+W} - \Phi^{\prime\prime}\at{R_0+W} V
\end{align*}
satisfies 
\begin{align}\label{eqn:Fclaim}
    \|U\|_{Z_0} \in \Do{\|V\|_Y}.
\end{align}
For that we write $\tilde{U}(x)= \exp(\eta x) U (x)$ for $\eta=-\eta_-, \eta_+$. Without restriction $\|V\|_Y \leq 1$ which implies that the range of $R_0(.)+W(.) + \theta(.) V(.)  $ is contained in a fixed compact interval $I$ for any choice of $\theta(.) \in [0,1]$.  Then using H\"older continuity of $\Phi^{\prime\prime}$ on $I$ and a pointwise Taylor expansion we obtain with 
\begin{align}
\tilde{U}&:=\exp( \eta .) 
    \bigl(\Phi^\prime\at{R_0+W+V}-\Phi^\prime\at{R_0+W} - \Phi^{\prime\prime}\at{R_0+W} V\bigr)  \nn  \mbox{ such that} \\
    \|\tilde{U}\|_{\fspaceL^2}&=\| \exp( \eta .) 
    \bigl(\Phi^\prime\at{R_0+W+V}-\Phi^\prime\at{R_0+W} - \Phi^{\prime\prime}\at{R_0+W} V\bigr) 
    \|_{\fspaceL^2} \nn \\ & \leq \| \exp( \eta .)(
    \Phi^{\prime\prime}\at{R_0(.)+W(.)+ \theta(.) V(.)} V(.)- \Phi^{\prime\prime}\at{R_0(.)+W(.)} V(.)\bigr)    \|_{\fspaceL^2} \nn\\
     & \leq C  \|V\|_Y^\beta  \| V\|_Y=\Do{ \|V\|_Y} ,  \label{eqn:Ufinal}
\end{align}
which yields the desired estimate \eqref{eqn:Fclaim}. Hence $F$ is Fr\'echet differentiable with respect to $W$ using Proposition \ref{prop:weight}.

By similar arguments as in \eqref{eqn:Ufinal} 
$W \mapsto DH(W)\in B(Y,Z_0) $ is  continuous in $W$, such that $D_2 F(\eps,W)$  is continuous in $B(Y,Z)$ by Proposition \ref{prop:weight}. 

To show continuity of the derivative with respect to $\eps$ at $\eps=0$ and $W=0$,  we will use \eqref{eqn:expaeps2}. We need to estimate the fractional Sobolev norm $\|\Phi^{\prime\prime}( R_0) V\|_{\fspaceH^s_\eta(\Rset)}$ where we choose  $0<s<\beta/2<1$, where $\beta$ is the H\"older exponent of the  $\Phi^{\prime\prime}$. Some necessary details about the underlying unweighted spaces can be found in \cite{NPV12}. So we need to estimate  $\|\exp( \eta .)\Phi^{\prime\prime}( R_0) V \|_{\fspaceH^s}$ for  $\eta=-\eta_-, \eta_+$. We use the Kato-Ponce inequality \cite{GO14} to estimate  
\begin{align}\nn
   & \|D^s(\exp( \eta .)\Phi^{\prime\prime}( R_0) V )\|_{\fspaceL^2}\\
   \leq & C \left(\|D^s\Phi^{\prime\prime}( R_0)\|_{\fspaceL^\infty} \| \exp( \eta .) V(.) \|_{\fspaceL^2}
   +\|\Phi^{\prime\prime}( R_0)\|_{\fspaceL^\infty} \|D^s(\exp( \eta .) V(.) )\|_{\fspaceL^2} \right). \label{eqn:KP}
\end{align}
For $V \in Y$ with $\|V \|_Y \leq 1$, we obtain uniform bounds on $\| \exp( \eta .) V(.) \|_{\fspaceL^2} \leq \|V \|_Y  $ and $\|D^s(\exp( \eta .) V(.) )\|_{\fspaceL^2} \leq \| \exp( \eta .) V(.) \|_{\fspaceH^1} \leq \| V\|_Y $. As above we obtain bounds 
 on $\|\Phi^{\prime\prime}( R_0)\|_{\fspaceL^\infty}$ due to the boundedness of $R_0$ and the continuity of $\Phi^{\prime\prime}$ on the image of $R_0$. It remains to estimate the fractional derivative of $\Phi^{\prime\prime}( R_0)$, we use the H\"older continuity of $\Phi^{\prime\prime}$ to get a uniform bound of the form  $|\Phi^{\prime\prime}( R_0(x))- \Phi^{\prime\prime}( R_0(y)) |\leq C_\beta |x-y|^\beta$. By \cite[Prop. 3.4]{NPV12} we can use an equivalent representation of the fractional derivative: 
 \begin{align*}
     |D^s (\Phi^{\prime\prime}( R_0))(x)|  \leq & C 
     \int_\Rset \left| \frac{\Phi^{\prime\prime}( R_0(x))- \Phi^{\prime\prime}( R_0(y)) }{|x-y|^{1+2s}}\right|\dint{y} \\
      \leq & C \int_{|x-y|\leq 1} \frac{|\Phi^{\prime\prime}( R_0(x))- \Phi^{\prime\prime}( R_0(y)) |}{|x-y|^{1+2s}}\dint{y}  \\ & \quad + C  \int_{|x-y| >  1} \frac{|\Phi^{\prime\prime}( R_0(x))- \Phi^{\prime\prime}( R_0(y)) |}{|x-y|^{1+2s}}\dint{y} \\
      \leq & C \int_{|x-y|\leq 1} \frac{C_\beta |x-y|^\beta}{|x-y|^{1+2s}}\dint{y}   + C  \int_{|x-y| >  1} \frac{2\|\Phi^{\prime\prime}( R_0)\|_{\fspaceL^\infty}}{|x-y|^{1+2s}}\dint{y} < \infty,
 \end{align*}
 as $1+2s -\beta<1$ for the first integral and $1+2s>1$ for the second. This implies that the expression in \eqref{eqn:KP} is bounded.  With part (3) of Proposition \ref{prop:weight} we obtain continuity of $D_2 F(\eps,W)$ at $(0,0)$.
\end{proof}

\begin{lemma}\label{lem:inv} The operator $L V= V- a_0 \ast (\Phi^{\prime\prime}\at{R_0} V) $ has a bounded inverse in $B(Z,Y).$
\end{lemma}
\begin{proof}
    For $g \in Z$ we look for $V\in Y$ such that $LV=g$, this is equivalent to the ODE
    \begin{align}
        \label{eqn:ODE} V^\prime(x)=& -V(x)+ \Phi^{\prime\prime}(R_0(x)) V(x)+f(x)\\
        V(0)=&0 \nn
    \end{align}
    where the initial condition is due the requirement that $V \in Y $. The unique solution of \eqref{eqn:ODE} is given by 
    \begin{align}\label{eqn:ODEsol}
        V(x) = & \int_0^x \exp\Bigl(\int_s^x(-1+ \Phi^{\prime\prime}(R_0(\tau)) \dint{\tau}  \Bigr) f(s) \dint{s}.
    \end{align}
    The main task is to check that \eqref{eqn:ODEsol} defines an operator in $B(Z,Y)$.   We consider $\eta_+$ and integration for $x\geq 0$.
We first note that as $R_0 \to 0$ exponentially by ODE theory and as $\Phi^{\prime\prime}(0)=p_+$ we have
\begin{align*}
    0< \int_0^\infty (\Phi^{\prime\prime}(R_0(\tau))-p_+) \dint{\tau} < \frac{M}{2}.
\end{align*}
Hence we can start bounding parts of the $\fspaceL^2_\eta$-norm
 \begin{align*}
 &\int_0^\infty \exp(2\eta_+ x) |V(x)|^2 \dint{x}\\
     =&\int_0^\infty \exp( 2 \eta_+ x) \left( \int_0^x \exp\Bigl(\int_s^x(-1+ \Phi^{\prime\prime}(R_0(\tau)) \dint{\tau}  \Bigr) f(s) \dint{s} \right)^2\dint{x}\\
     =&\int_0^\infty \left( \int_0^x \exp\Bigl(\int_s^x(\eta_+ -(1 -p_+)  + (\Phi^{\prime\prime}(R_0(\tau))-p_+) \dint{\tau}  \Bigr) \exp(\eta_+ s) f(s) \dint{s} \right)^2\dint{x} \\ \leq& \exp(M)  
     \int_0^\infty \left( \int_0^x \exp\bigl( (\eta_+ -(1 -p_+)) (x-s) \bigr) \exp(\eta_+ s) f(s) \dint{s} \right)^2\dint{x}     
     \\=&\exp(M)  
     \int_0^\infty \left( \int_0^x \exp( -\delta (x-s))  \exp\bigl( (\eta_+ -(1 -p_+)+\delta) (x-s) \bigr) \exp(\eta_+ s) f(s) \dint{s} \right)^2\!\!\!\dint{x}. 
\end{align*}
 Then we  use the Cauchy-Schwarz for the inner integral, where $0<\delta< (1-p_+)- \eta_+$, and Fubini's Theorem, to estimate 
  \begin{align*}
     \leq & \exp(M)  
     \int_0^\infty \left( \int_0^x \exp( -2 \delta (x-s)) \dint{s}  \right)\\& \qquad \qquad \times  \left(\int_0^x \exp\bigl( 2(\eta_+ -(1 -p_+)+\delta) (x-s) \bigr) \exp(2\eta_+ s) f^2(s) \dint{s} \right) \dint{x}   \\ \leq& \frac{\exp(M)}{2\delta }    \int_0^\infty   \int_s^\infty  \exp\bigl( 2(\eta_+ -(1 -p_+)+\delta) (x-s) \bigr) \exp(2\eta_+ s) f^2(s)  \dint{x} \dint{s}
     \\=& \frac{\exp(M)}{2\delta }   \int_0^\infty   \int_s^\infty  \exp\bigl( 2(\eta_+ -(1 -p_+)+\delta) (x-s) \bigr) \dint{x} \exp(2\eta_+ s) f^2(s) \dint{s} \\ =&  \frac{\exp(M)}{2\delta (2(\eta_+ -(1 -p_+)+\delta))} \int_0^\infty  \exp(2\eta_+ s) f^2(s) \dint{s} \leq C 
      \|f \|^2_{\fspaceH^1_{\eta_+}}.   \end{align*}
The integrals for $\eta_-$ and $x<0$ follow analogously. Due to $\sup_{x \in \Rset} |\Phi^{\prime\prime}(R_0(x)) | <\infty $ and \eqref{eqn:ODE} we also obtain the same weighted $\fspaceL^2$ estimates for $V'$. Combining all of these yields the desired bound 
\begin{align}
    \label{eqn:invest}
    \|V\|_Y \leq C \|f\|_Z,
\end{align}
such that $L$ has a bounded inverse.
\end{proof}

\subsection{Existence of fronts}\label{subsec:fronts}

Now we can complete the existence part of Main Theorem. 
\begin{proof}[Proof of existence part of Theorem \ref{existthm}] 
The existence part of Theorem \ref{existthm} follows from the Implicit Function Theorem.  The operator $F$ as \eqref{eq:Fsep} satisfies assumptions Theorem \ref{thm:IFT}: The mapping properties and the continuity assumption (1) follow from Lemmas \ref{lem:F1} and \ref{lem:F2}. Condition (2) at the point 
$(0,0)$ follows by Lemma \ref{lem:F=0}, condition (3) by Lemma \ref{lem:FisFrech} and condition (4) by Lemma \ref{lem:inv}. Then the operator $G$ in Theorem \ref{thm:IFT} yields $R_\eps= R_0+ G(\eps)$. With 
$G(\eps)\in \fspaceH^1_{\eta_+} \cap \fspaceH^1_{-\eta_-}$, this implies non-optimal exponential decay of 
$|R_\eps -r_{\pm}|$ for $x \to \pm \infty$.

The implicit function argument only gives continuity of $R_\eps$ with respect to $\eps$. Inspecting them as fixed points gives the quantitative estimate by subtracting $R_0= a_0 \ast \Phi'(R_0)$ from \eqref{eq:fixedpt}: 
\begin{align*}
    R_\eps- R_0 =& a_\eps \ast \Phi^\prime(R_\eps)-a_0 \ast \Phi^\prime(R_0)\\
    =& a_0 \ast(\Phi^\prime(R_\eps)-\Phi^\prime(R_0))+   (a_\eps- a_0) \ast \Phi'(R_\eps).
\end{align*}
By expressing $(\Phi^\prime(R_\eps)-\Phi'(R_0))(.) = \Phi^{\prime\prime}\at{R_0(.)+\theta(.) (R_\eps-R_0)(.)} (R_\eps(.)-R_0(.))$ and denoting $f_\eps= (a_\eps- a_0) \ast \Phi'(R_\eps)$, this is equivalent to the ODE
\begin{align}
    \label{eqn:ODEdifference}
    (R_\eps- R_0)'(x)+ \bigl(1 - \Phi^{\prime\prime}\at{R_0(x)+\theta(x) (R_\eps-R_0)(x)}\bigr) (R_\eps-R_0)(x) = f_\eps(x).
\end{align}
This has the same structure as \eqref{eqn:ODE} in the proof of Lemma \ref{lem:inv}.  We  observe that Proposition \ref{prop:weight} (2) implies $\|f_\eps\|_{\fspaceH^1} \leq C \eps$.  
The desired estimate \eqref{H1est} then follows by 
the same argument as in the proof of Lemma \ref{lem:inv}.
\end{proof}

The monotonicity and decay properties in Theorem \ref{existthm} require a more careful analysis, which we will provide in the next section.

\section{Optimal decay rates}
In this section we denote by $R_\eps$ the nonlinear wave from Theorem \ref{existthm} that connects the states $1$ and $0$. The function
\begin{align*}
P_\eps:=\Phi^{\prime\prime}\at{R_\eps}\,,
\end{align*}
attains the asymptotic values 
\begin{align*}
P_\eps\at{x}\quad \xrightarrow{\;x\to\pm\infty\;}\quad p_\pm, 
\end{align*}
see \eqref{AsympStatesP} which satisfy
\begin{align*}
0\leq p^+<1<p^-<\infty
\end{align*}
due to the strict convexity of $\Phi^\prime$ on $[0,1]$. 
Moreover, our non-optimal decay results from \S\ref{subsec:fronts} together with the H\"older continuity of $\Phi^{\prime\prime}$ imply the existence of positive constants $C_0^\pm>0$  --- which do not depend on $\eps$ --- such that
\begin{align}
\label{AP.DecayPminus}
\babs{P_\eps\at{x}-p^-} &\leq C|R_\eps(x) -1|^\beta \leq C_0^-\exp\at{\beta\eta^-x}\\ \label{AP.DecayP}
\babs{P_\eps\at{x}-p^+} & \leq C|R_\eps(x) -0|^\beta \leq  C_0^+\exp\at{-\,\beta\eta^+ x}
\end{align}
holds for all $x\in\Rset$ provided that $\eps$ is sufficiently small. We recall that  $\eta^\pm$ are chosen arbitrarily in the open intervals
\begin{align*}
\eta^-\in\bpair{0}{p^--1}\,,\qquad \eta^+\in\bpair{0}{1-p^+}.
\end{align*}
%
%
%
We further consider the function
\begin{align}
\label{AP.DefS}
S_\eps\at{x}:=-\tfrac{\dint}{\dint{x}}\,R_\eps\at{x}
\end{align}
which fulfills
\begin{align}
\label{AP.NormalS}
\int\limits_\Rset S_\eps\at{x}\dint{x}=1\,,
\qquad S_\eps\at{0}\quad\xrightarrow{\;\;\eps\to0\;\;}\quad S_0\at{0}=-R^\prime_0\at{0}>0
\end{align}
(thanks to \eqref{eqn:ODEdifference})
as well as the linear integral equation
\begin{align}
\label{AP.Eqn1}
S_\eps=a_\eps\ast\bat{P_\eps\,S_\eps}.
\end{align}
Our goal is to prove that $S_\eps$ is nonnegative and decays with certain decay rates for all sufficiently small $\eps>0$.  We 
expect that $S_\eps$ decays like
\begin{align}
\label{AP.Decay1}
S_\eps\at{x}\sim\exp\at{+\mu_\eps^-\,x}\quad\text{for}\quad x<0
\end{align}
and
\begin{align}
\label{AP.Decay2}
S_\eps\at{x}\sim\exp\at{-\mu_\eps^+\,x}\quad\text{for}\quad x>0\,,
\end{align}
where $\pm \frac{i}{2\pi}\,\mu_\eps^\pm$ are bounded roots of $\hat{a}_\eps^{-1}-p_\pm$ as $\eps \to 0$.   
In particular, the expected decay rates for $\eps>0$ are not given by  $p^-$ and $p^+$, which describe the tail decay of the solutions to the asymptotic ODE \eqref{normcontlim},  but the nonlocal terms in \eqref{norminteqtw} give rise to corrections terms of order $O\at{\eps^2}$ as 
specified below in \eqref{Lem.AP.ModKernels.E3} and  \eqref{Lem.AP.ModKernels.E4}.

\subsection{Auxiliary equations}

The analysis of the decay rates at $\pm \infty$ can be separated from each other by adapting the operators. The key observation in this context is that \eqref{AP.Eqn1} can also be written as
\begin{align}
\label{AP.Eqn23}
S_\eps={a}^+_\eps\ast\bat{\at{P_\eps-p_+}S_\eps}
\qquad \text{and}\qquad S_\eps=-{a}^-_\eps\ast\bat{\at{P_\eps-p_-}S_\eps}
\end{align}
%
where the modified convolution kernels $a_\eps^-$ and $a_\eps^+$ correspond to the symbol functions
\begin{align}
\label{AP.DefModKernels}
\hat{a}_\eps^\pm\at{k}
=
 \frac{\pm 1}{\hat{a}_\eps^{-1}-p^{\pm}}
=\frac{\pm\sinc^2\at{\eps\,
\pi\,k}}{1-p_\pm \sinc^2\at{\eps\,\pi\,k}+2\,\pi\,\iu\,k\,\sinc^2\at{\eps\,\pi\,k}}\,
\end{align}
with poles at $k=\pm \frac{i}{2\pi}\,\mu_\eps^\pm$
(see Proposition \ref{Lem.AP.ModKernels} below).
The numerator in \eqref{AP.DefModKernels} decays quadratically while the denominator has no real zeros and is asymptotically constant. In particular, the convolution kernel $a_\eps^-$ and $a_\eps^+$ are well-defined, sufficiently smooth and satisfy
\begin{align*}
\hat{a}_\eps^\pm\at{k}\quad\xrightarrow{\;\;\eps\to0\;\;} \quad \hat{a}_0^\pm\at{k}=\frac{\pm1}{\;1-p^\pm+2\,\pi\,\iu\,k\;}
\end{align*}
where the pointwise limit  functions correspond to the nonnegative but discontinuous convolution kernels
\begin{align*}
a_0^-\at{x}=\left\{\begin{array}{cc}\mhexp{\at{p^--1}\,x}&\text{for $x\leq0$\,,}\\
0&\text{for  $x>0$\,,}\end{array}\right.\qquad
a_0^+\at{x}=\left\{\begin{array}{cc}0&\text{for $x<0$\,,}\\
\mhexp{-\at{1-p^+}\,x}&\text{for  $x\geq0$\,.}\end{array}\right.
\end{align*}
\begin{proposition}[properties of the modified kernels]
\label{Lem.AP.ModKernels}
For any sufficiently small $\eps$ exists positive real number $\mu_\eps^-$, $\mu_\eps^+$ and
$\nu_\eps^-$, $\nu_\eps^+$
with
\begin{align}
\label{Lem.AP.ModKernels.E3}
\mu_\eps^\pm=\pm \at{1-p^\pm}\mp\eps^2\,\tfrac{1}{12}\,(1+p^\pm)\,\at{1-p^\pm}^2+O\at{\eps^4}
\end{align}
and
\begin{align}
\label{Lem.AP.ModKernels.E4}
\nu_\eps^\pm= 1 -\eps^2\,\tfrac{1}{12}\,\at{1+ p^\pm}\,\at{1-\,p^\pm}+O\at{\eps^4}
\end{align}
such that the representation formulas
\begin{align}
\label{Lem.AP.ModKernels.E2}
\hat{a}_\eps^\pm\at{k}=\hat{b}_\eps^\pm\at{k}+\hat{c}_\eps^\pm\at{k}\,,\qquad \quad
\hat{b}_\eps^\pm\at{k}:=\frac{\pm\nu_\eps^\pm}{\pm\mu_\eps^\pm+2\,\pi\,\iu\,k}
\end{align}
hold for all $k\in\Rset$. Here, the symbol functions  $\hat{c}_\eps^\pm$ are holomorphic in the open complex strip  
\begin{align*}
\Omega=\big\{ z\in\Cset\,:\babs{\mhIm\at{z}}<M \big\},
\end{align*}
where $M> 1+ \frac{1}{ \pi}\max(|\mu_\eps^-|,\,|\mu_\eps^+|)$, and satisfy for $s\in [0,1] $
\begin{align}
    \label{eqn:modkernelsup}
    \|\hat{c}^\pm_\eps(.) (1+ |.|^s) \|_{\fspaceL^\infty(\Omega) }\leq C_s \eps^{1-s}.
\end{align}
\end{proposition}
The proof is similar to the proof of Proposition \ref{prop:symbol} and will be also given in section \ref{Sec:Four}.
The monomial term $\hat{b}_\eps^\pm$ in \eqref{Lem.AP.ModKernels.E2} represents a simple and purely imaginary pole of $a_\eps^\pm$ at $k=\mp\mu_\eps^\pm/\at{2\pi\iu}$ while all other complex poles of $a_\eps^\pm$ exhibit a nonvanishing real part as well as an imaginary part with rather large modulus. Moreover, the kernel functions $b_\eps^\pm$ satisfy
\begin{align*}
\pm\,\mu_\eps^\pm b_\eps^\pm\at{x}+\tfrac{\dint}{\dint{x}}b_\eps^\pm\at{x}=\pm\,\nu_\eps^\pm \delta_0\at{x}\,,
\end{align*}
with $\delta_0$ being the standard Dirac distribution, and can hence be regarded as a multiple of the fundamental solution to the differential operator $\pm\mu_\eps^\pm+\tfrac{\dint}{\dint{x}}$. 
Consequently, by substituting the splitting \eqref{Lem.AP.ModKernels.E2} into \eqref{AP.Eqn23}, we obtain that the function $S_\eps$ from \eqref{AP.DefS} satisfies the representation formulas
\begin{align}
\label{AP.ODE1}
-\mu_\eps^- S_\eps+\tfrac{\dint}{\dint{x}}S_\eps=\nu_\eps^-\nat{P_\eps-p^-} S_\eps-F_\eps^-
\end{align}
and
\begin{align}
\label{AP.ODE2}
\mu_\eps^+ S_\eps+\tfrac{\dint}{\dint{x}}S_\eps=\nu_\eps^+\nat{P_\eps-p^+} S_\eps+F_\eps^+,
\end{align}
where
\begin{align}
\label{AP.DefF}
F_\eps^\pm=\at{\pm\mu_\eps^\pm +
\tfrac{\dint}{\dint{x}}}\bigl(c_\eps^\pm\ast
\bat{\nat{P_\eps-p^\pm}\,S_\eps}\,\bigr)
\end{align}
represent nonlocal perturbations. 
Both terms on the right-hand side are small for $x \to +\infty$ in \eqref{AP.ODE1} and for  $x \to -\infty$ in \eqref{AP.ODE2}. 
Therefore based on the leading linear ODE on the left-hand side we can expect the behavior in \eqref{AP.Decay1} and \eqref{AP.Decay2}. 
  This heuristic explanation will be made precise in the next subsection.

\subsection{Improvement of exponential rates}

\begin{lemma}[conditional improvement of exponential decays rates]
\label{Lem.AP.Rates} Let $\theta$ be sufficiently small such that 
$ \frac{\iu ( \pm\mu^\pm + \theta)}{2 \pi}  \in \Omega$ as in Proposition \ref{Lem.AP.ModKernels} 
and $\theta < \min( \eta^-, \eta^+)$. 
Suppose that there exists a constant $D_\eps>0$ as well as positive decay rates $\la_\eps^-$, $\la_\eps^+$ with
\begin{align*}
0<\la_\eps^\pm<\mu_\eps^\pm\,, \qquad \la_\eps^\pm\neq \mu_\eps^\pm- \beta\theta\,, \qquad \la_\eps^\pm\neq \mu_\eps^\pm- \tfrac{\beta}{2}\theta
\end{align*}
such that
\begin{align}
\label{Lem.AP.Rates.E1}
\babs{S_\eps\at{x}}\leq D_\eps\,\min\left\{\exp\at{+\la_\eps^-\,x}\,,\;\exp\at{-\la_\eps^+\,x}\right\}
\end{align}
holds for all $x\in\Rset$. Then we also have
that
\begin{align}
\label{Lem.AP.Rates.E2}
\babs{S_\eps\at{x}}\leq \tilde{D}_\eps\,\min\left\{\exp\at{+\tilde{\la}_\eps^-\,x}\,,\;\exp\at{-\tilde{\la}_\eps^+\,x}\right\}
\end{align}
with
\begin{align}
\label{Lem.AP.Rates.E3}
\tilde{\la}_\eps^\pm=\min\left\{\mu_\eps^\pm\,,\;\la_\eps^\pm+\frac{\beta}{2}\theta\right\}
\end{align}
for some constant $\tilde{D}_\eps$ and  all $x\in\Rset$.
\end{lemma}
\begin{proof}
From \eqref{AP.DecayPminus}, \eqref{AP.DecayP} and \eqref{Lem.AP.Rates.E1} we infer that each of the two estimates
\begin{align}\label{eqn:Lem14est1}
\babs{\nat{P_\eps\at{x}-p^\pm}\,S_\eps\at{x}}\leq C\,D_\eps\,\exp\at{\mp\nat{\la_\eps^\pm+\beta \theta}x}
\end{align}
holds for all $x\in\Rset$. For the error term $F_\eps^\pm$ it is most convenient to use weighted $\fspaceL^2$ estimates and to obtain these via the Fourier transform, here we first obtain some bounds thanks to  \eqref{eqn:modkernelsup} for $s=1$, we will show later that the terms are actually small in $\eps$. 
\begin{align} \nonumber
&\| F_\eps^\pm\at{. }\exp\at{\pm\nat{\la_\eps^\pm+ \frac{\beta}{2} \theta} .} \|_{\fspaceL^2(\Rset)}\\  \nonumber =&
\left( \int_\Rset  \left( \at{\pm\mu_\eps^\pm +
\tfrac{\dint}{\dint{x}}}\bigl(c_\eps^\pm\ast
\bat{\nat{P_\eps-p^\pm}\,S_\eps}\,\bigr) (x) \right)^2 \exp\at{\pm 2 \nat{\la_\eps^\pm+ \tfrac{\beta}{2} \theta} x}  \dint{x} \right) ^{1/2}
\\ \nonumber
=&
\left( \int_{\Rset+ \tfrac{\pm \iu }{2 \pi}  (\la_\eps^\pm+ \tfrac{\beta}{2} \theta)}  \left| \at{\pm\mu_\eps^\pm + \iu 2 \pi  k} \hat{c}_\eps^\pm(k) 
\widehat{\nat{P_\eps-p^\pm}\,S_\eps} (k)\, \right|^2  \dint{k} \right) ^{1/2} \\ \nonumber
\leq & C \|\hat{c}^\pm_\eps(.) (1+ |.|) \|_{\fspaceL^\infty(\Omega) }^{1/2}  
\left( \int_{\Rset+ \tfrac{\pm \iu}{2 \pi}  (\la_\eps^\pm+ \tfrac{\beta}{2} \theta)}  \left|
\widehat{\nat{P_\eps-p^\pm}\,S_\eps} (k)\, \right|^2  \dint{k} 
\right) ^{1/2}\\  \nonumber\leq &
\,C\, C_1 \left( \int_\Rset  \left( 
\bat{\nat{P_\eps-p^\pm}\,S_\eps}\,(x) \right)^2 \exp\at{\pm2 \nat{\la_\eps^\pm+ \tfrac{\beta}{2} \theta} x}  \dint{x} \right) ^{1/2}
\\
\leq & \,C\, C_1 D_\eps \left( \int_\Rset  \left( 
\bat{P_\eps-p^\pm} (x) \right)^2 \exp\at{\pm 2 \nat{ \tfrac{\beta}{2} \theta} x}  \dint{x} \right) ^{1/2}. \label{eqn:Fbounde}
\end{align}
We note for the $F^+_\eps$ case using \eqref{AP.DecayP} and  H\"older's inequality  with exponents $\tfrac{1}{\beta}$ and $\tfrac{1}{1-\beta}$ after splitting the integral as well as using the non-optimal exponential decay of $R_\eps$ from Section \ref{subsec:fronts}
\begin{align*}
   & \int_\Rset  \left( 
\bat{P_\eps-p^+} (x) \right)^2 \exp\at{ 2 \nat{ \tfrac{\beta}{2} \theta} x}  \dint{x}\leq \int_\Rset  \left( 
C |R_\eps|^\beta (x) \right)^2 \exp\at{ 2 \nat{ \tfrac{\beta}{2} \theta} x}  \dint{x}
\\=& \int_{-\infty}^0   \left( 
C |R_\eps|^\beta (x)  \right)^2 \exp\at{ 2 \nat{ \tfrac{\beta}{2} \theta} x}  \dint{x}+ \int_0^\infty \left( C
 |R_\eps|^\beta (x) \right)^2 \exp\at{ 2 \nat{ \tfrac{\beta}{2} \theta} x}  \dint{x}\\
\leq & C^2  \|R_\eps\|^{2\beta}_{\fspaceL^\infty(\Rset)} 
\int_{-\infty}^0 \exp\at{ 2   \nat{ \tfrac{\beta}{2} \theta} x}  \dint{x}\\
&\quad +  C^2
\left( \int_0^\infty 
 |R_\eps|^2(x) \exp\at{ 2  \theta x} \dint{x} \right)^{\beta} \left( \int_0^\infty   \exp\at{ - 2 \tfrac{1}{1-\beta} \nat{ \tfrac{\beta}{2} \theta} x} \dint{x} \right)^{1-\beta} \\
 \leq &\tilde{C}_\eps\Bigl(\frac{1}{\beta \theta} (\|R_0\|_{\fspaceL^\infty(\Rset)} + \|R_0-R_\eps\|_{{\fspaceL^\infty(\Rset)}})  +\left(\frac{1-\beta}{\beta \theta}\right) ^{1-\beta}  \left( \int_0^\infty 
 \exp\at{ -2 \eta_+ x}  \exp\at{ 2  \theta  x} \dint{x} \right)^{\beta} \Bigr),  
\end{align*}
which is bounded by \eqref{H1est} as already proved in Section \ref{subsec:fronts}. We  deal  similarly with the $F^-_\eps$ case. Altogether we there exists a $C>0$ such that
\begin{align}\label{eqn:Fepsest}
    \| F_\eps^\pm\at{. }\exp\at{\pm\nat{\la_\eps^\pm+ \frac{\beta}{2} \theta} .} \|_{\fspaceL^2(\Rset)} \leq C D_\eps.
\end{align}
With the ODE representation in \eqref{AP.ODE2} and 
%
%
the Duhamel principle 
\begin{align*}
    &S_\eps(x)\\ =&S_\eps(0) \exp\bat{-\mu_\eps^+\,x} +
 \exp\bat{-\mu_\eps^+\,x}\int_0^x   \exp\bat{\mu_\eps^+\,y} (\nu_\eps^+\nat{P_\eps-p^+} S_\eps(y) +F^+_\eps(y))
    \dint{y} 
\end{align*}
yields with \eqref{eqn:Lem14est1}  the estimate 
\begin{align*}
    |S_\eps(x)| &\leq|S_\eps(0)| \exp\bat{-\mu_\eps^+\,x} \\&\quad+
 \exp\bat{-\mu_\eps^+\,x}  \int_0^x   \exp\bat{\mu_\eps^+\,y} \nu_\eps^+ C\,D_\eps\,\exp\at{-\nat{\la_\eps^+ +\beta \theta}y} \dint{y}  \\
 &\quad  +\exp\bat{-\mu_\eps^+\,x}  \int_0^x  \exp\bat{(\mu_\eps^+- \nat{\la_\eps^+ +\tfrac{\beta}{2} \theta} )\,y}  \exp\at{\nat{\la_\eps^+ +\tfrac{\beta}{2}\theta}y} |F^+_\eps(y)|     \dint{y}.  
\end{align*}
The final term can be estimated using H\"older's inequality and \eqref{eqn:Fepsest}
\begin{align*}
    &\int_0^x  \exp\bat{(\mu_\eps^+- \nat{\la_\eps^+ +\tfrac{\beta}{2} \theta} )\,y}  \exp\at{\nat{\la_\eps^+ +\tfrac{\beta}{2} \theta}y} |F^+_\eps(y)|     \dint{y}\\
    \leq&  \left( \int_0^x  \exp\bat{2(\mu_\eps^+- \nat{\la_\eps^+ +\tfrac{\beta}{2} \theta} )\,y}  \dint{y}\right)^{1/2} \left(\int_0^x   \exp\at{2 \nat{\la_\eps^+ +\tfrac{\beta}{2} \theta}y} |F^+_\eps(y)|^2     \dint{y}\right)^{1/2}\\
    \leq &\left(\frac{\exp\bat{2(\mu_\eps^+- \nat{\la_\eps^+ +\tfrac{\beta}{2}\theta} )\,x} +1}{2\babs{\mu_\eps^+-\at{\la_\eps^++\tfrac{\beta}{2} \theta}}}\right)^{1/2} CD_\eps \\ \leq &\frac{1}{\babs{\mu_\eps^+-\at{\la_\eps^++\tfrac{\beta}{2} \theta}}^{1/2}} (\exp\bat{(\mu_\eps^+- \nat{\la_\eps^+ +\tfrac{\beta}{2} \theta} )\,x} +1) CD_\eps. 
\end{align*}
Summarizing the previous estimates gives for $S_\eps$ in the interval $\oointerval{0}{\infty}$:
\begin{align*}
\babs{S_\eps\at{x}}\leq& \at{\babs{S_\eps\at{0}}+
\frac{\nu_\eps^+\,D_\eps}{\;\babs{\mu_\eps^+-\at{\la_\eps^++\beta \theta}}\;}
+ \frac{CD_\eps }{\babs{\mu_\eps^+-\at{\la_\eps^++\tfrac{\beta}{2} \theta}}^{1/2}}
}\,\exp\bat{-\mu_\eps^+\,x}\;\\&+\left(\frac{\nu_\eps^+\,D_\eps}{\;\babs{\mu_\eps^+-\at{\la_\eps^++\beta \theta}}\;}
+ \frac{CD_\eps }{\babs{\mu_\eps^+-\at{\la_\eps^++\tfrac{\beta}{2}\theta}}^{1/2}}\right)\,\exp\bat{-\at{\la_\eps^++\tfrac{\beta}{2} \theta}\,x}
\end{align*}
for all $x>0$ and applying similar arguments to \eqref{AP.ODE1} we justify
\begin{align*}
\babs{S_\eps\at{x}}\leq& \at{\babs{S_\eps\at{0}}+\frac{\nu^-_\eps D_\eps}{\;\babs{\mu_\eps^--\at{\la_\eps^-+\beta \theta}}}\;+\frac{C D_\eps }{\;\babs{\mu_\eps^--\at{\la_\eps^-+\tfrac{\beta}{2} \theta}}^{1/2}\;}}\,\exp\bat{+\mu_\eps^-\,x}\;\\&+ \at{\frac{\nu^-_\eps D_\eps}{\;\babs{\mu_\eps^--\at{\la_\eps^-+\beta \theta}}}\;+\frac{C D_\eps }{\;\babs{\mu_\eps^--\at{\la_\eps^-+\tfrac{\beta}{2} \theta}}^{1/2}\;}}\,\,\exp\bat{+\at{\la_\eps^-+\tfrac{\beta}{2}\theta}\,x}
\end{align*}
for all $x<0$. The claims \eqref{Lem.AP.Rates.E2} and \eqref{Lem.AP.Rates.E3} now follow immediately provided that $\tilde{D}_\eps$ is chosen appropriately.
\end{proof}
\begin{corollary}[optimal decay rates]
\label{Cor.AP.Decay}
For all sufficiently small $\eps>0$, we have
\begin{align*}
\babs{S\at{x}}\leq D_\eps\,\min\left\{\exp\at{+\mu_\eps^-\,x}\,,\;\exp\at{-\mu_\eps^+\,x}\right\}
\end{align*}
for some constant $D_\eps$.
\end{corollary}
\begin{proof}
Our nonlinear existence result in the proof so far  implies the existence of nonoptimal but positive decay rates $-\eta^-$ and $\eta^+$ and without loss of generality from \eqref{Lem.AP.ModKernels.E3} we can assume both $\eta^\pm<\mu_\eps^\pm$ and choose $\theta$ such that $2\at{\mu_\eps^\pm-\eta^\pm}/(\beta\theta)\notin\Nset$. The claim now follows by applying Lemma \ref{Lem.AP.Rates} inductively
with $\lambda_\eps^\pm=\eta^\pm+ j \tfrac{\beta}{2} \theta$ with $j \in \Nset$.
\end{proof}
Before we can show that the decay rates in Corollary \ref{Cor.AP.Decay} we need to improve the estimate  \eqref{eqn:Fepsest}.
\begin{lemma}\label{lem:Fepssmall}
There exists $\delta>0$ and constants $\ol{D}_\eps^\pm$ such that 
\begin{align}\label{eqn:lemFepssmall}
\int_0^\infty\babs{F_\eps^\pm\at{y}}^2  \exp\bat{\pm 2\at{\mu_\eps^\pm +\delta}\,y} \dint{y} \leq \ol{D}_\eps^\pm  \eps^{\beta/4}\, 
\end{align}
\end{lemma}
\begin{proof}
We start as in \eqref{eqn:Fbounde}
with $\la_\eps^\pm= \mu_\eps^\pm$ and add a small extra weight $\delta$ to be chosen later.
We use \eqref{eqn:modkernelsup} with $s=1-\beta/4$
\begin{align} \nonumber
&\| F_\eps^\pm\at{. }\exp\at{\pm\nat{\mu_\eps^\pm+ \delta} .} \|_{\fspaceL^2(\Rset)}
\\ \nonumber
=&
\left( \int_{\Rset+ \tfrac{\pm \iu }{2 \pi}  (\mu_\eps^\pm+ \delta)}  \left| \at{\pm\mu_\eps^\pm + \iu 2 \pi  k} \hat{c}_\eps^\pm(k) 
\widehat{\nat{P_\eps-p^\pm}\,S_\eps} (k)\, \right|^2  \dint{k} \right) ^{1/2} \\ \nonumber
\leq & C \|\hat{c}^\pm_\eps(.) (1+ |.|^{1-\beta/2}) \|_{\fspaceL^\infty(\Omega) } ^{1/2} 
\left( \int_{\Rset+ \tfrac{\pm \iu}{2 \pi}  (\mu_\eps^\pm+ \delta)} (1+ |k|^{\beta/2}) \left|
\widehat{\nat{P_\eps-p^\pm}\,S_\eps} (k)\, \right|^2  \dint{k} 
\right) ^{1/2}\\  \leq &
\,C\, \eps^{\beta/4} \left( \int_\Rset  \left((1+ D^{\beta/4})\left(
\bat{\nat{P_\eps-p^\pm}\,S_\eps}\,(.) \exp\at{\pm \nat{\mu_\eps^\pm+ \delta}(.)} \right)\right)^2    \dint{x} \right) ^{1/2}.
\label{eqn:Fsmall}
\end{align} 
For notational convenience, we give details for the $+$ case, the other case follows in parallel. As in Lemma \ref{lem:FisFrech} we use the 
Kato-Ponce inequality \cite{GO14} 
\begin{align}\nn
   & \|D^{\beta/4}\left(\bat{\nat{P_\eps-p^+}\,S_\eps}\,(.) \exp\at{ \nat{\mu_\eps^++ \delta}(.)} \right)\|_{\fspaceL^2}\\ \nn=&
   \|D^{\beta/4}\left(\left[\bat{P_\eps-p^+} \,(.) \exp\at{ 2 \delta(.)} \right] \left[
     S_\eps\,(.) \exp\at{ \nat{\mu_\eps^+- \delta}(.)}  \right]\right)\|_{\fspaceL^2}
   \\
   \leq & C \Bigl(\|D^{\beta/4}\left[\bat{P_\eps-p^+} \,(.) \exp\at{ 2 \delta(.)} \right]\|_{\fspaceL^\infty} \| S_\eps\,(.) \exp\at{ \nat{\mu_\eps^+- \delta}(.)} \|_{\fspaceL^2} \nn \\& \quad
   +\|\bat{P_\eps-p^+} \,(.) \exp\at{ 2 \delta(.)}\|_{\fspaceL^\infty} \|D^{\beta/4}(S_\eps\,(.) \exp\at{ \nat{\mu_\eps^+- \delta}(.)} )\|_{\fspaceL^2} \Bigr). \label{eqn:KP2}
\end{align}
Taking the $\fspaceL^2$ norm of the ODE \eqref{AP.ODE2} and using the bounds on $S_\eps$ in Corollary  \ref{Cor.AP.Decay} together with \eqref{eqn:Fepsest} yield $\fspaceH^1$ bounds on $S_\eps\,(.) \exp\at{ \nat{\mu_\eps^+- \delta}(.)}$ for any small $\delta>0$. Thus both terms involving $S_\eps$ in \eqref{eqn:KP2} are bounded. 
 We already obtained bounds in \eqref{AP.DecayP} 
 on 
 \begin{align}
     \|\bat{P_\eps-p^+} \,(.)  \exp\at{ 2 \delta(.)}\|_{\fspaceL^\infty} \leq C_\delta \mbox{ for } 0< \delta \leq \frac{\beta \eta^+}{2} \label{eqn:Peps+}
 \end{align}
  It remains to estimate the fractional derivative of 
\begin{align*}
   \bat{P_\eps-p^+} \,(.)  \exp\at{ 2 \delta(.)}=
   ( \Phi^{\prime\prime}( R_\eps) -p^+) \,(.)  \exp\at{ 2 \delta(.)}. 
\end{align*}
  We use the H\"older continuity of $\Phi^{\prime\prime}$  and
  the bounds on $R'_\eps =S_\eps$ in Corollary  \ref{Cor.AP.Decay} to get a uniform bound for $0< c \leq \beta \min(\mu_\eps^-, \mu_\eps^+)$ and $|x-y|\leq 1$ of the form  
  \begin{align}\label{eqn:PepsHolder}|\Phi^{\prime\prime}( R_\eps(x))- \Phi^{\prime\prime}( R_\eps(y)) |\leq 
  C_\beta |R_\eps(x)-R_\eps(y)|^\beta \leq C \exp \at{- c |x|} |x-y|^\beta. \end{align} 
  By \cite[Prop. 3.4]{NPV12} we can use an equivalent representation of the fractional derivative: 
 \begin{align*}
     &|D^{\beta/4} [ (\Phi^{\prime\prime}( R_\eps) -p^+) \,(.)  \exp\at{ 2 \delta(.)}](x)|  \\
     \leq & C 
     \int_\Rset \left| \frac{(\Phi^{\prime\prime}( R_\eps) -p^+) \,(x)  \exp\at{ 2 \delta x}- (\Phi^{\prime\prime}( R_\eps) -p^+) \,(y)  \exp\at{ 2 \delta y }}{|x-y|^{1+\beta/2}}\right|\dint{y} \\
      \leq & C \int_{|x-y|\leq 1} \frac{(\Phi^{\prime\prime}( R_\eps) -p^+) \,(x)  \exp\at{ 2 \delta x}- (\Phi^{\prime\prime}( R_\eps) -p^+) \,(y)  \exp\at{ 2 \delta y }}{|x-y|^{1+2s}}\dint{y}  \\ & \quad + C  \int_{|x-y| >  1} \frac{|(\Phi^{\prime\prime}( R_\eps) -p^+) \,(x)  \exp\at{ 2 \delta x }- (\Phi^{\prime\prime}( R_\eps) -p^+) \,(y)  \exp\at{ 2 \delta(y)}|}{|x-y|^{1+\beta/2}}\dint{y} \\
      \leq & C \int_{|x-y|\leq 1} \frac{
     | (\Phi^{\prime\prime}( R_\eps)(x)- \Phi^{\prime\prime}( R_\eps)(y)) \exp\at{ 2 \delta x } | }{|x-y|^{1+\beta/2}}\dint{y} \\ &\quad+ 
C \int_{|x-y|\leq 1} \frac{ |(\Phi^{\prime\prime}( R_\eps) -p^+) \,(x) ( \exp\at{ 2 \delta x}-\exp\at{ 2 \delta y})   | }{|x-y|^{1+\beta/2}}\dint{y}     
     \\ &\quad + C  \int_{|x-y| >  1} \frac{2\|\bat{P_\eps-p^+} \,(.)  \exp\at{ 2 \delta(.)}\|_{\fspaceL^\infty}}{|x-y|^{1+\beta/2}}\dint{y},
 \end{align*}
 which is bounded by  \eqref{eqn:Peps+} for the third integral,
 by \eqref{eqn:Peps+} and  $|\exp\at{ 2 \delta x}-\exp\at{ 2 \delta y}| \leq C |x-y| \exp \at{ 2 \delta x} $ for the second integral and by \eqref{eqn:PepsHolder} for the first integral,  all for some suitable $\delta>0$. 
  This implies that the expression in \eqref{eqn:KP2} is bounded and we obtain \eqref{eqn:lemFepssmall} for the $+$ case, the estimate for the $-$ case uses the same steps.
\end{proof}
\begin{theorem}[exponential upper and lower bounds for $S_\eps$]\label{thm:rates}
There exists two constants $0<\ul{D}<1<\ol{D}<\infty$ such that the pointwise estimates
\begin{align*}
\ul{D}\leq S_\eps\at{x}\, \exp\at{-\mu_\eps^-\,x} \leq \ol{D}\qquad\text{for}\quad x<0
\end{align*}
and
\begin{align*}
\ul{D}\leq S_\eps\at{x}\, \exp\at{+\mu_\eps^+\,x} \leq \ol{D}\qquad\text{for}\quad x>0
\end{align*}
are satisfied for all sufficiently small $\eps$.
\end{theorem}
\begin{proof}
By Corollary \ref{Cor.AP.Decay}, the quantities
\begin{align*}
\ul{D}_\eps^+=\inf_{x>0}\babs{S_\eps\at{x}}\,\exp\at{+\mu_\eps^+\,x}\,,\qquad
\ol{D}_\eps^+=\sup_{x>0}\babs{S_\eps\at{x}}\,\exp\at{+\mu_\eps^+\,x}
\end{align*}
are well-defined real numbers and the representation formula \eqref{AP.ODE2} implies
\begin{align*}
S_\eps\at{x}=&\exp\at{-\mu_\eps^+\,x}\,\exp\at{\nu^+_\eps \int\limits_{0}^x \bat{P_\eps\at{s}-p^+}\dint{s}}\,S_\eps\at{0}+\\&
\int\limits_0^x \exp\bat{-\mu_\eps^+\,\at{x-y}}\,\exp\at{\nu^+_\eps\int\limits_{y}^x \bat{P_\eps\at{s}-p^+}\dint{s}}F_\eps^+\at{y}\dint{y}
\end{align*}
thanks to the Duhamel principle, where the  properties of $P_\eps$ --- due to the assumptions on $\Phi$  and \eqref{AP.DecayP} --- ensure
\begin{align*}
0<\nu^+_\eps\int\limits_{y}^x \bat{P_\eps\at{s}-p^+}\dint{s}\leq C_P\qquad \text{for all}\quad 0<y<x<\infty
\end{align*}
for some constant $C_P$ independent of $\eps$.
Combining these with the results of Corollary \ref{Cor.AP.Decay} and Lemma \ref{lem:Fepssmall} we obtain
\begin{align*}
\babs{S_\eps\at{x}}\leq \exp\at{-\mu_\eps^+\,x}\,\exp\at{C_P}\at{
S_\eps\at{0}+\eps^{\beta/4}\,\ol{D}_\eps^+\,C_0\left(\int\limits_0^x\exp\at{-2 \delta\,y}\dint{y}\right)^{1/2}}
\end{align*}
and hence
\begin{align*}
\babs{S_\eps\at{x}}\,\exp\at{+\mu_\eps^+\,x}\leq \exp\at{C_P}\,\at{
S_\eps\at{0}+\frac{\;\eps^{\beta/4}\,\ol{D}_\eps^+\,C_0}{\sqrt{2\delta}}\;}\,.
\end{align*}
Taking the supremum over $x>0$ on the left hand side and rearranging terms reveals that
\begin{align*}
\ol{D}_\eps^+\leq \frac{\exp\at{C_P}}{\;\;\D1-\frac{\;\eps^{\beta/4}\,C_0\,\exp\at{C_P}\;}{\sqrt{2\delta}}\;\;}\,S_\eps\at{0}\leq 2\,\exp\at{C_P}\,S_\eps\at{0}
\end{align*}
holds for all sufficiently small $\eps$. On the other hand, as long as $S_\eps\at{y}\geq 0$ holds for all $y\in\oointerval{0}{x}$, the estimate
\begin{align*}
\mu_\eps^+\,S_\eps\at{y}+\tfrac{\dint}{\dint{x}}S_\eps\at{y}\geq-\abs{F^+_\eps\at{y}}
\end{align*}
holds due to  \eqref{AP.ODE2}.
The comparison principle for ODEs, the Duhamel principle and H\"older's inequality yield
\begin{align*}
&S_\eps\at{x} \geq \exp\at{-\mu_\eps^+\,x}S_\eps\at{0}
-  \exp\at{-\mu_\eps^+\,x} \int_0^x   \exp\at{\mu_\eps^+\,y} |F^+_\eps(y)| \dint{y}\\
\geq& \exp\at{-\mu_\eps^+\,x}S_\eps\at{0}\\ &
-  \exp\at{-\mu_\eps^+\,x} \left(
\int_0^x\babs{F_\eps^+\at{y}}^2  \exp\bat{2\at{\mu_\eps^++\delta}\,y} \dint{y} \right)^{1/2} \left(\int\limits_0^x\exp\at{-2\delta\,y}\dint{y}\right)^{1/2} \\ \geq &
\exp\at{-\mu_\eps^+\,x}\at{S_\eps\at{0}
-\frac{\;\eps^{\beta/4}\,\ol{D}_\eps^+\,C_0\;}{\sqrt{2\delta}}}
\end{align*}
in this case and for $\eps$ sufficiently small
\begin{align*}
S_\eps\at{x}>0\quad\text{ for all}\quad x>0\,,\qquad \quad
\ul{D}_\eps^+\geq \tfrac12\, S_\eps\at{0}\,.
\end{align*}
Moreover, repeating all arguments for $x<0$ we derive
\begin{align*}
\ol{D}_\eps^-\geq \tfrac12\, S_\eps\at{0}
\,,\qquad \quad
S_\eps\at{x}>0\quad\text{ for all}\quad x<0\,,\qquad \quad
\ul{D}_\eps^-\leq \tfrac12\,\exp\at{C_P}\, S_\eps\at{0}
\end{align*}
so the claim follows in view of \eqref{AP.NormalS}.
\end{proof}

\begin{proof}[Proof of the remaining parts of Theorem \ref{existthm}] 
The existence part of Theorem \ref{existthm} was given at the end of section \ref{sec:IFT}.  Monotonic decay follows from the positivity of $S_\eps= R_\eps'$ as provided by the lower bounds in Theorem \ref{thm:rates}. Integrating $S_\eps$ from $\pm \infty $ yields the decay behavior.
\end{proof}

\section{Proofs of estimates on Fourier symbols}
\label{Sec:Four}

To study the Fourier symbols $\hat{a}_\eps$, $\hat{a}^+_\eps$ and $\hat{a}^-_\eps$ jointly, we consider 
\[ \hat{a}^\mu_\eps\at{k}= \frac{\sinc^2(\eps \pi k)}{1- \mu \sinc^2(\eps \pi k) + 2 \pi \iu k \sinc^2(\eps \pi k),}\]
where $\mu \in \{0,p_-,p_+ \}$. We change variable $z := \eps \pi k$ and rewrite
\begin{align*}
    A^\mu_\eps(z) =& \frac {\eps \sin^2(z)}{D^\mu_\eps(z)} \mbox { where}\\
    & D^\mu_\eps(z)= \eps z^2 - \eps \mu \sin^2(z) + 2 \iu z \sin^2(z).
    \end{align*} 
We collect some basic properties of $D^\mu_\eps$
\begin{lemma}\label{lem:pole}
    There exists $\eps_0>0$ such that for $0<\eps<\eps_0$ $D^\mu_\eps$ has a double zero at $0$ and unique simple zero $z_\eps$
 in the complex ball $B_{ 0.9 \pi}(0)$. Furthermore, the zero $z_\eps$ satisfies 
 \begin{align}
     \label{eqn:zeps} z_\eps=& \frac{\iu \eps}{2}(1-\mu) +  O\at{\eps^2},\\
     (D^\mu_\eps)'(z_\eps)=&-\frac{\iu \eps^2 }{2} (1-\mu)^2+ O\at{\eps^3}, \\
     \frac{\eps \sin^2(z_\eps)}{ (D^\mu_\eps)'(z_\eps)}=&\frac{\eps}{2 \iu}+  O\at{\eps^2}. \label{eqn:pole3}
 \end{align}
 \end{lemma}
\begin{proof}
For $\eps=0$, $D^\mu_\eps$ has a triple zero at $z=0$ and no further zeros in $B_{ 0.9 \pi}(0)$. Using Rouch\'e's theorem and as the double zero at $0$ persists, the additional zero bifurcates by standard bifurcation theory. The formulas follow by Taylor expansion.    
\end{proof}
Then  the leading pole of $A^\mu_\eps$ is given by 
\begin{align}
    \label{eqn:Beps} B^\mu_\eps(z)=  \frac{\eps \sin^2(z_\eps)}{ (D^\mu_\eps)'(z_\eps) (z-z_\eps)} 
\end{align}
and the corresponding  function in $k$ is denoted by
\begin{align}  \label{eqn:beps}
    \hat{b}_\eps^\mu(k)=  \frac{\sin^2(z_\eps)}{(D^\mu_\eps)'(z_\eps) \bigl( \pi  k- \tfrac{z_\eps}{\eps}\bigr) }
\end{align}
Then $B^\mu_\eps$ is a good approximation of $A^\mu_\eps$ on $B_{ 0.8 \pi}(0)$.
\begin{lemma}\label{lem:bulkest}
 There exists $\eps_0>0$ such that for all $0<\eps<\eps_0$ the holomorphic function   $G^\mu_\eps(z):= A^\mu_\eps (z)- B^\mu_\eps(z)$ satisfies
  \[\sup_{ z \in B_{ 0.8 \pi}(0)} | G^\mu_\eps(z)| \leq C \eps^2. \]
\end{lemma}
\begin{proof}
The function $ A^\mu_\eps $ is meromorphic on $B_{ 0.9 \pi}(0)$ by Lemma \ref{lem:pole} with a simple pole at $z_\eps$, hence $G^\mu_\eps$ is holomorphic by construction. The estimate follows by   Cauchy's integral formula on $\partial B_{ 0.9 \pi}(0)$. On this contour $z \sin^2(z)$ is uniformly bounded away from $0$. Then  on  $\partial B_{ 0.9 \pi}(0)$
\begin{align*}
A^\mu_\eps(z)= \frac{\eps \sin^2(z)}{\eps z^2 - \eps \mu \sin^2(z) + 2 \iu z \sin^2(z)}= \eps \frac{ sin^2(z)}{ 2 \iu z \sin^2(z)(1+ O\at{\eps} ) }
= \frac{\eps}{2 \iu z} + O\at{\eps^2}, \end{align*}
while by \eqref{eqn:pole3} 
\begin{align*}
    B^\mu_\eps(z)=  \frac{\eps}{2 \iu (z-z_\eps)} + O\at{\eps^2}= \frac{\eps}{2 \iu z} + O\at{\eps^2} \mbox{ on } \partial B_{ 0.9 \pi}(0),
\end{align*}
such that  
\begin{align*}
    A^\mu_\eps(z)- B^\mu_\eps(z)=  O\at{\eps^2} \mbox{ on } \partial B_{ 0.9 \pi}(0).
\end{align*}
Then finally using Cauchy's integral formula
\begin{align*}
    G^\mu_\eps(z)= \frac{1}{2 \pi \iu}\int_{\partial B_{ 0.9 \pi}(z_\eps)}  \frac{A^\mu_\eps(\xi)- B^\mu_\eps(\xi)}{\xi-z } \dint{\xi} =\frac{1}{2 \pi \iu}\int_{\partial B_{ 0.9 \pi}(z_\eps)}  \frac{O\at{\eps^2}}{\xi-z } \dint{\xi} \leq C \at{\eps^2}
\end{align*}
as required. 
\end{proof}
Changing back to $k$ coordinates, Lemma \ref{lem:bulkest} implies an estimate on a ball 
with radius proportional to $\eps^{-1}$
\begin{align} \label{eqn:bulkkest}
    \sup_{k = \frac{z}{\eps \pi} \in B_{0.8/ \eps}(0) } | \hat{a}^\mu_\eps(k)-  \hat{b}_\eps^\mu(k)| \leq C \eps^2.
\end{align}
We control the behavior outside the ball by showing  decay in $k$.
\begin{lemma}\label{lem:A.tailest} Let  $\eps$ be sufficiently small and let $M>0$ be fixed. Then for all $z \in \Cset$ with $|\Re(z)|\geq 0.8 \pi$ and $| \Im (z)|\leq M \pi  \eps$ the estimate
\[ \left| A^\mu_\eps(z) \right| \leq \frac{C \eps}{|z|}\]
  holds.
\end{lemma}
\begin{proof}
    We need to show that the expression
\begin{align} \label{eqn:aim}
    \frac{z A^\mu_\eps  (z)}{\eps}= \frac{\sin^2(z)}{\eps z - \eps \mu \tfrac{\sin^2(z)}{z}+2  \iu \sin^2(z)}
\end{align}    
is bounded for $z \in \Omega=\{ z \in \Cset \mid |\Re(z)|\geq 0.8 \pi, | \Im (z)|\leq M \pi  \eps\} $.  We write $z=x+ \iu y$ with $x,y \in \Rset$ and use $\sin(z)=\sin(x) \cosh (y) + \iu \cos(x) \sinh(y)$.  Then we  observe  for all 
$\eps$ small enough and
$z \in \Omega$
\begin{align}
 |\eps z| &\geq 0.8 \pi\, \eps\label{eqn:tail0}\\
    \left| \eps z- \eps \mu \frac{\sin^2(z)}{z}  \right|  & \leq (1 +|\mu| ) \eps \left|    z \right|\label{eqn:tail1}\\
 \left| \Im( \sin^2(z))\right|^2  &=4\left|  \sin(x) \cosh (y) \cos(x) \sinh(y) \right|^2 \leq C \eps^2  |\sin(x)|^2 \label{eqn:tail2}\\
 \left| \sin^2(z)\right|^2& = \left| \sin^2(x) \cosh^2 (y) - \cos^2(x) \sinh^2(y) \right|^2  +   \left| \Im( \sin^2(z))\right|^2  \nonumber \\
 &\leq  C \left( |\sin(x)|^4+  \eps^4 + \eps^2  |\sin(x)|^2  \right)  \label{eqn:tail3}
\end{align}
We can consider \eqref{eqn:aim} in three cases:
\paragraph{Case 1: $\frac{|\eps z|}{2} \geq 3 |\sin^2(z)|$.} Then in the denominator 
\begin{align*}
    |\eps z - \eps \mu \tfrac{\sin^2(z)}{z}+2  \iu \sin^2(z)|&\geq |\eps z| - \left(2+\frac{\eps \mu}{0.8  \pi}\right)    |\sin^2(z)| \geq \frac{|\eps z|}{2}
\end{align*}
for $\eps$ sufficiently small
and while the numerator can be estimated by the assumption, we obtain 
\begin{align*}
   \left|\frac{z A^\mu_\eps  (z)}{\eps} \right| \leq \frac{\tfrac{1}{6} | \eps z|}{\tfrac{1}{2 }|\eps z|} = \frac{1}{3}
\end{align*}
as required.
\paragraph{Case 2: $   |\sin^2(z)| \geq  (1 +|\mu|) |\eps z| $.} Then with  \eqref{eqn:tail1}
\begin{align*}
    |\eps z - \eps \mu \tfrac{\sin^2(z)}{z}+2  \iu \sin^2(z)|&\geq 2  |\sin^2(z)| 
    - (1 +|\mu|) |\eps z|  
    \geq    |\sin^2(z)|
\end{align*}
and 
\begin{align*}
   \left|\frac{z A^\mu_\eps  (z)}{\eps} \right| \leq 1
\end{align*}
as required.
\paragraph{Case 3: $  \frac{1}{1+|\mu|}  \leq  \frac{|\eps z| }{ |\sin^2(z)| }   \leq 6 $.}
The remaining intermediate case requires a more detailed argument, which uses the idea that idea both $z$ and $\sin(z)$ are mainly real in this case. We first obtain with \eqref{eqn:tail0} that 
\[ |\sin^2(z)|^2 \geq C \eps^2.  \]
Using \eqref{eqn:tail3} then implies
\[ |\sin(x)| \geq \tilde{C} \sqrt{\eps},\]
such that for $\eps$ sufficiently small
\begin{eqnarray}
\nonumber
\frac{\left|\Re (\sin^2(z))\right|}{ \left| \Im( \sin^2(z))\right|}&=&
\frac{\left|  \sin^2(x) \cosh ^2(y) -\cos^2(x) \sinh^2(y) \right| }{2 \left|  \sin(x) \cosh (y) \cos(x) \sinh(y) \right|}\\
\nonumber
&=&
\frac{\left|  \sin^2(x) \cosh ^2(y) + O\at{\eps^2} \right| }{2 \left|  \sin(x) \cosh (y) \cos(x) \sinh(y) \right|}\\
\nonumber
&\geq &
\frac{\left|  \sin(x) \right| \left| \cosh(y)\right| }{4 \left|  \cos(x) \sinh(y) \right|}\\
\label{eqn:cas31}
&\geq &\frac{\tilde{C}}{4M  \pi\sqrt{\eps}}\gg 1
\end{eqnarray}
using $|\tanh{y}|\leq |y| \leq M  \pi{\eps}$.
Furthermore due to the form of $\Omega$ we have
\begin{align}\label{eqn:cas32}|\Re (\eps z) | \geq 0.8  \pi\eps \gg M \pi \eps^2  \geq | \Im (\eps z)|.\end{align} 
As  $\eps z$ and $\sin^2(z)$ are comparable in this case, we also have 
\begin{align}
   \label{eqn:cas33}|\Re (\eps z) |&\gg  | \Im (\sin^2(z))|\\
   \label{eqn:cas34}|\Re (\sin^2(z)) |& \gg | \Im (\eps z)|.
\end{align}
Then \eqref{eqn:cas31},\eqref{eqn:cas32}, \eqref{eqn:cas33}  and  \eqref{eqn:cas34} together imply for $\eps$ sufficiently small 
\begin{align*}
    &|\eps z - \eps \mu \tfrac{\sin^2(z)}{z}+2  \iu \sin^2(z)|\\
    = &\sqrt{|\Re (\eps z-\eps \mu \tfrac{\sin^2(z)}{z}  ) -2 \Im (\sin^2(z)) |^2 + 
    |2\Re (\sin^2(z))  + \Im  (\eps z-\eps \mu \tfrac{\sin^2(z)}{z}  )|^2 } \\
    \geq&  \frac{1}{2}  \sqrt{|\Re(\eps z)|^2 +4 |\Re(\sin^2(z))|^2 } \geq  |\Re(\sin^2(z))|  \\
        \geq&  \frac{1}{2}  |\sin^2(z)|, 
\end{align*}
such that finally 
\begin{align*}
   \left|\frac{z A^\mu_\eps  (z)}{\eps} \right| \leq \frac{|\sin^2(z)|}{\tfrac{1}{2} |\sin^2(z)|} = 2.
\end{align*}
Combining all three cases concludes the proof. 
\end{proof}

We are now in the position to prove Propositions \ref{prop:symbol} and \ref{Lem.AP.ModKernels}.

\begin{proof}[Proof of Proposition \ref{prop:symbol}]
 The Fourier symbol $\hat{a}_\eps$ is a quotient of two entire functions. The poles within the  large ball  $B_{0.8/ \eps}(0) $ have been identified in Lemma \ref{lem:pole} for $\eps$ sufficiently small, such that there are 
 no  poles in $\Rset \times \iu [-\Breve{\eta}_-,\Breve{\eta}_+]  \cap  B_{0.8/ \eps}(0) $, hence $\hat{a}_\eps$ is 
 holomorphic. There are no poles in $\Rset \times \iu [-\Breve{\eta}_-,\Breve{\eta}_+]  \setminus  B_{0.8/ \eps}(0) $ by Lemma \ref{lem:A.tailest} either, showing that $\hat{a}_\eps$ is holomorphic  in the strip $\Rset \times \iu [-\Breve{\eta}_-,\Breve{\eta}_+]$.    

The $\fspaceL^\infty$-bound  \eqref{eqn:fourlinftyest} follows directly by boundedness of $\hat{a}_\eps$ on $\Rset \times \iu [-\Breve{\eta}_-,\Breve{\eta}_+]$ and the decay estimate of Lemma \ref{lem:A.tailest}. 

For the next two estimates we will first establish that for $\hat{b}_\eps$ as in \eqref{eqn:beps}, we 
have  a uniform constant $C$ such that for all $k \in \Rset \times \iu [-\Breve{\eta}_-,\Breve{\eta}_+]$
\begin{align} \label{eqn:ba0}
    |\hat{b}_\eps(k) - \hat{a}_0 (k)| \leq  C \eps \frac{1}{1 +|k|}.
\end{align}
For $k$ in any fixed ball $B$ around $0$, we get from \eqref{eqn:bulkkest}, that $\hat{b}_\eps$ is close to 
$\hat{a}_\eps$, while $\hat{a}_0$  is close to  $\hat{a}_\eps$  uniformly on  $B$ by Taylor expanding $\sinc^2(\eps \pi k)=1 + O\at{\eps^2} $, hence \eqref{eqn:ba0} holds on the unit ball.
While a direct calculation yields for $k$ outside the unit ball and $\eps$ sufficiently small
\begin{align*}
    |\hat{b}_\eps(k) - \hat{a}_0 (k)| &=
     \frac{\eps \sin^2(z_\eps)}{(2 \eps z_\eps+  2 \iu \sin^2(z_\eps) + 2 \iu z_\eps \sin(z_\eps) \cos(z_\eps)) ( \pi k -(z_\eps/ \eps)) }  - \frac{1}{1+2 \pi \iu k} \\
&=    \frac{\pi k (2 \eps z_\eps + 4 \iu z_\eps \sin(z_\eps) \cos(z_\eps)) + O\at{\eps^2} }{(2 \eps z_\eps+  2 \iu \sin^2(z_\eps) + 2 \iu z_\eps \sin(z_\eps) \cos(z_\eps)) ( \pi k -(z_\eps/ \eps) (1+2 \pi \iu k)}.
\end{align*}
We observe the following  uniform estimates in $k$  
\begin{align*}
    &(2 \eps z_\eps + 4 \iu z_\eps \sin(z_\eps) \cos(z_\eps))\in O\at{\eps^3},\\
    &|D'_\eps (z_\eps)| = | (2 \eps z_\eps+  2 \iu \sin^2(z_\eps) + 2 \iu z_\eps \sin(z_\eps) \cos(z_\eps))| \geq C \eps^2,\\
    &(z_\eps/ \eps) = \iu/2 + O\at{\eps},
\end{align*}
which together yield \eqref{eqn:ba0} outside the unit ball too.

Then \eqref{eqn:fourerror} restricted to $k \in \Rset \times \iu [-\Breve{\eta}_-,\Breve{\eta}_+]  \cap  B_{0.8/ \eps}(0) $ follows from Lemma \ref{lem:bulkest} and the previous estimate \eqref{eqn:ba0}. The estimate in the tail follows from Lemma \ref{lem:A.tailest} and using $|k| >0.8/\eps$
\begin{align*}
    |\hat{a}_\eps(k) - \hat{a}_0 (k)| \leq |\hat{a}_\eps(k)| +  |\hat{a}_0 (k)| \leq \frac{C}{|k|} = C \eps. 
\end{align*}

The interpolated weight estimate \eqref{eqn:fourerror2} follows 
on
$\Rset \times \iu [-\Breve{\eta}_-,\Breve{\eta}_+]  \cap  B_{0.8/ \eps}(0) $ from \eqref{eqn:fourerror}: The maximal weight is $C/\eps^{1-s}$
which leads to an extra factor $C \eps^{s-1}$, which gives together with the $\eps$ in  \eqref{eqn:fourerror} the required estimate on $\Rset \times \iu [-\Breve{\eta}_-,\Breve{\eta}_+]  \cap  B_{0.8/ \eps}(0) $.  While on  $\Rset \times \iu [-\Breve{\eta}_-,\Breve{\eta}_+]  \setminus B_{0.8/ \eps}(0) $ 
we use \eqref{eqn:fourlinftyest} where we also gain a factor $|k|^{-s} \leq C \eps^s$.

 Finally, the $\fspaceL^2$-estimate \eqref{eqn:fourL2exp} follows directly by boundedness of $\hat{a}_\eps$ and the decay estimate in Lemma \ref{lem:A.tailest}.
\end{proof}

\begin{proof}[Proof of Proposition \ref{Lem.AP.ModKernels}]  The arguments follow in parallel to the above proof of Proposition \ref{prop:symbol}. The expansion of the coefficients $\mu_\eps^\pm=\mp \frac{2i z_\eps^\pm}{\eps}$ and $\nu_\eps^\pm = \frac{2 \iu \sin^2(z^\pm_\eps)}{(D^\pm_\eps)'(z_\eps)}$ are obtained using Lemma \ref{lem:pole}. By construction $\hat{a}^\pm_\eps$ and $\hat{b}^\pm_\eps$
have the same poles in $B_{0.8/ \eps}(0)$. Both have  no further poles in $\Omega$ such that   
$\hat{c}^\pm_\eps=\hat{a}^\pm_\eps-\hat{b}^\pm_\eps$ is holomorphic on $\Omega$. 

We are setting $z= \pi k \eps$. The bulk estimate of \eqref{eqn:modkernelsup} for $\Omega \cap B_{0.8/ \eps}(0)$ follows from Lemma  \ref{lem:bulkest} for $\mu= p_\pm$ uniformly for all $s\in[0,1]$ as the maximal weight is  $C/\eps$. 
On  $\Omega \setminus B_{0.8/ \eps}(0) $, 
we again set $z= \pi k \eps$ such that \ref{lem:A.tailest} yields a bound $C/|k|$.  Then for a weight $1+|.|^s$ for fixed $s\in[0,1]$  we  gain a factor $|k|^{1-s} \leq C \eps^{1-s}$ as required.
\end{proof}

\bibliographystyle{alpha}
\bibliography{lit_paper.bib}

\end{document}